\documentclass{article}
\usepackage{amsmath, amsfonts, amssymb, amsthm, color}

\newcommand{\HH}{\mathbb{H}}
\newcommand{\ZZ}{\mathbb{Z}}
\newcommand{\EE}{\mathbf{E}}
\newcommand{\BB}{\mathbf{B}}
\newcommand{\CC}{\mathbb{C}}
\newcommand{\RR}{\mathbb{R}}

\DeclareMathOperator{\Per}{Per}
\DeclareMathOperator{\per}{per}
\DeclareMathOperator{\Det}{Det}

\DeclareMathOperator{\Alt}{Alt}
\DeclareMathOperator{\Sym}{Sym}
\DeclareMathOperator{\Mat}{Mat}

\DeclareMathOperator{\sgn}{sgn}

\newcommand{\frakg}{\mathfrak{g}}
\newcommand{\frakk}{\mathfrak{k}}
\newcommand{\frakp}{\mathfrak{p}}

\newcommand{\frakm}{\mathfrak{m}}
\newcommand{\frakh}{\mathfrak{h}}

\newcommand{\fraku}{\mathfrak{u}}
\newcommand{\frako}{\mathfrak{o}}
\newcommand{\fraks}{\mathfrak{s}}
\newcommand{\fraksp}{\mathfrak{sp}}
\newcommand{\frakusp}{\mathfrak{usp}}
\newcommand{\frakgl}{\mathfrak{gl}}

\newcommand{\PD}{\mathcal{PD}}

\newcommand{\Symm}{{\mathfrak{S}}}
\newcommand{\imgunit}{\sqrt{-1}}
\newcommand{\veps}{\varepsilon}
\newcommand{\talpha}{\tilde{\alpha}}
\newcommand{\tbeta}{\tilde{\beta}}
\newcommand{\tXi}{\tilde{\Xi}}

\newcommand{\T}{\hspace*{.1ex}{}^t\hspace*{-.1ex}}
\newcommand{\comb}[2]{{\mathcal{I}^{#1}_{#2}}}
\newcommand{\bcomb}[2]{{\BAR{\mathcal{I}}^{#1}_{#2}}}
\newcommand{\BAR}[1]{\overline{#1}}
\newcommand{\dblprime}{'\hspace*{-.2ex}'}

\newcounter{thmenum}
\newenvironment{thmenumerate}{%
\begin{list}{$(\thethmenum)$}{%
\usecounter{thmenum}
\setlength{\labelsep}{.5em}
\setlength{\labelwidth}{-7pt}
\setlength{\topsep}{0pt}
\setlength{\partopsep}{0pt}
\setlength{\parsep}{0pt}
\setlength{\leftmargin}{1pt}
\setlength{\rightmargin}{0pt}
\setlength{\itemindent}{\leftmargin}
\setlength{\itemsep}{0pt}
}}
{\end{list}}

\newcommand{\bracketcite}[1]{{[\refcite{#1}]}}
\renewcommand\bracketcite{\cite}

\theoremstyle{theorem}
\newtheorem{theorem}{Theorem}[section]
\newtheorem{proposition}{Proposition}[section]
\newtheorem{lemma}{Lemma}[section]
\theoremstyle{defintion}
\newtheorem{definition}{Definition}[section]
\newtheorem{remark}{Remark}[section]

\begin{document}

\title{
A note on the Capelli identities for symmetric pairs of Hermitian type
}
\author{Kyo Nishiyama and Akihito Wachi}



\date{}
\maketitle

\begin{abstract}
We get several identities of differential operators in determinantal form.  
These identities are non-commutative versions of the formula of Cauchy-Binet or 
Laplace expansions of determinants, and if we take principal symbols, they are reduced to such classical formulas.  
These identities are naturally arising from the generators of the rings of invariant differential operators over symmetric spaces, 
and have strong resemblance to the classical Capelli identities.  
Thus we call those identities \emph{the Capelli identities for symmetric pairs}.
\end{abstract}

\def\thefootnote{\relax}
\footnote{keywords: Capelli identity, symmetric space, invariant differential operators, dual pair, Weil representation.}

\section{Introduction}

In this article we give several identities of differential operators
associated to see-saw pairs of reductive Lie groups.
These identities look similar to 
the classical Capelli identities (\bracketcite{capelli87:_uber,capelli90:_sur}): 
\begin{gather*}
\det(E_{i,j} + (n-j)\delta_{i,j}) = 
\det(x_{i,j}) \det(\partial_{i,j}) =
\det(E'_{i,j} + (j-1)\delta_{i,j}),
\\
\text{where~}
E_{i,j} = \sum_{k=1}^n x_{k,i} \partial_{k,j}, \quad
E'_{i,j} = \sum_{k=1}^n x_{j,k} \partial_{i,k}, \quad
\partial_{i,j} = \partial/\partial x_{i,j}.
\end{gather*}
%
The Capelli identities give two different determinantal expressions of the same differential operator, 
which can be interpreted as 
the image of the centers of universal enveloping algebras
corresponding to a dual pair of reductive Lie groups 
(see \bracketcite{capelli90:_sur, MR1116239, MR2139918}).  
In this setting, the original Capelli's identity corresponds to the dual pair $ GL_n(\CC) \times GL_m(\CC) $, 
and recently, there appear many kinds of generalizations in addition to 
the pioneering work of \bracketcite{capelli90:_sur, MR1116239}.  
See, for example, the references 
\bracketcite{MR99a:17018, MR2001d:17009, 
MR2000433, MR2031449, MR2139918, 
molev_nazarov99:_capel_lie, 
MR1912651, 
wachi00:_capel_verma_i,MR2048468}. 

Our identities here also have a similar interpretation.  
In fact, they can be considered as different expressions of the same differential operator 
as the image of invariant differential operators on two symmetric spaces associated to a see-saw pair.  
Thus we call our identities 
\emph{the Capelli identities for symmetric pairs}.  
Note that 
our differential operators do \emph{not necessarily} come from
the center of the universal enveloping algebras, and yet they have a determinantal expressions.  
This is a remarkable difference of our identities from the classical Capelli identities.

To explain our identities more precisely, we need some notation.  
Let 
$\fraks_0 = \mathfrak{sp}_{2N}(\RR) $ 
be a real symplectic Lie algebra, 
and let $\frakg_0, \frakk_0, \frakm_0 $ and $ \frakh_0$ be
real reductive Lie subalgebras of $\fraks_0$.  
We assume that they form a see-saw pair:
\begin{align*}
\begin{array}{cccccccc}
\frakg_0  &  {}                     &  \frakm_0  \\
\cup      &  \text{\LARGE$\times$}  &  \cup      \\
\frakk_0  &  {}                     &  \frakh_0
\end{array},
\text{~where~}
\left\{
\begin{array}{l}
\text{$(\frakg_0, \frakk_0)$ and $(\frakm_0, \frakh_0)$
  are symmetric pairs}, \\
\text{$\frakg_0 \leftrightarrow \frakh_0$ and
  $\frakm_0 \leftrightarrow \frakk_0$ 
  are dual pairs in $\fraks_0$}.
\end{array}
\right.
\end{align*}
Here $\frakg_0 \leftrightarrow \frakh_0$ is called a \emph{dual pair}
if they are mutual commutants 
in the symplectic Lie algebra $\fraks_0$,
and $(\frakg_0, \frakk_0)$ is called a \emph{symmetric pair}
if $\frakk_0$ is a fixed-point subalgebra of $\frakg_0$
under a non-trivial involution.

Recall the symplectic Lie algebra $ \fraks_0 $ has the \emph{Weil representation} $\omega$ 
(or also called the \emph{oscillator representation}) 
acting on the polynomial ring $\CC[V]$ over a Lagrangian vector space $V$ 
(see \bracketcite{MR986027, MR0463359}).  
In fact, this is a representation of the two-fold double cover of 
the symplectic group $ Sp_{2N}(\RR) $, which is called a metaplectic group.  
But we only consider the infinitesimal version of this, realized as a Harish-Chandra module.  
Thus the complexifications of Lie algebras in see-saw pairs
act on $\CC[V]$ through the (differential of) the Weil representation $\omega$.  

We denote the complexification of $\frakg_0$ by $\frakg$ etc.,
and let $\PD(V)$ be the ring of differential operators 
with polynomial coefficients on a vector space $V$.
Then we have the following picture:
\begin{align*}
U(\frakg) 
\overset{\omega}{\longrightarrow}
\PD(V)
\overset{\omega}{\longleftarrow} 
U(\frakm).
\end{align*}
Let $K $ and $ H$ be the complex Lie groups
corresponding to $\frakk$ and $\frakh$ respectively.
Thanks to the very definition of the dual pair $(\frakg_0, \frakh_0)$,
the image of $U(\frakg)$ is $H$-invariant.
Similarly the image of $U(\frakm)$ is $K$-invariant.
Thus we have the following picture
by restricting to invariant subalgebras:
\begin{align}
U(\frakg)^K
\overset{\omega}{\longrightarrow}
\PD(V)^{K\times H}
\overset{\omega}{\longleftarrow} 
U(\frakm)^H.
\end{align}
Due to the result of Howe \bracketcite{MR986027} ,
both $U(\frakg)^K$ and $U(\frakm)^H$ are mapped \emph{onto} 
$\PD(V)^{K\times H}$.
In particular, we have $\omega(U(\frakg)^K) = \omega(U(\frakm)^H)$.

In our previous paper \bracketcite{nishiyama_lee_wachi}
we studied the case 
where $(\frakg_0, \frakk_0)$ is a Hermitian symmetric pair,
and $\frakm_0$ is a compact Lie algebra.
Let $\frakg_0 = \frakk_0 \oplus \frakp_0$ 
be a Cartan decomposition of $\frakg_0$.
The algebra of $K$-invariants $S(\frakp)^K$ of 
the symmetric algebra $S(\frakp)$ 
is finitely generated.  
We take $K$-invariant elements
$X_d \in U(\frakg)^K$ $(d = 1,2, \ldots)$
whose principal symbols are the generators of $S(\frakp)^K$.  
Then $\omega(X_d)$ is in $\PD(V)^{K \times H}$,
and there exists an inverse image in $U(\frakm)^H$
thanks to $\omega(U(\frakg)^K) = \omega(U(\frakm)^H)$.
In \bracketcite{nishiyama_lee_wachi},
we determined an inverse image $C_d \in U(\frakm)^H$ satisfying
\begin{equation}
\label{abstract.Capelli.identity}
\omega(X_d) = \omega(C_d)
\qquad
(X_d \in U(\frakg)^K, \; C_d \in U(\frakm)^H) .
\end{equation}
We call this formula a 
\emph{Capelli identity for a symmetric pair},
and $X_d$ and $C_d$ \emph{Capelli elements}.

In this article, we flip the role of symmetric pairs
$(\frakg_0, \frakk_0)$ and $(\frakm_0, \frakh_0)$,
and establish the similar identities.  

Namely, let $(\frakm_0, \frakh_0)$ be a Hermitian symmetric pair,
and assume $\frakg_0$ to be compact.
Let $\frakg_0 = \frakk_0 \oplus \frakp_0$ be
the $(\pm 1)$-eigenspace decomposition with respect to the involution.
We then take $K$-invariant elements $X_d \in U(\frakg)^K$
whose principal symbols are generators of $S(\frakp)^K$,
and try to find $H$-invariant elements $C_d \in U(\frakm)^H$ satisfying 
\eqref{abstract.Capelli.identity}.  
There are three see-saw pairs which fit into our setting:
\begin{gather}
\nonumber
\text{see-saw pairs with $\frakm_0$ Hermitian type,
  $\frakg_0$ compact}
\\
\label{table:see-saw-pairs}
\begin{array}{c@{\;}|@{\;\;}l@{\;\;}c@{\;\;}c@{\;\;}c@{\;\;}c}
  & \hfil \fraks_0 & \frakg_0 & \frakk_0 & \frakm_0 & \frakh_0 
\\ \hline
\text{Case $\RR$} 
& \fraksp_{2mn}(\RR) 
& \fraku_m & \frako_m(\RR) & \fraksp_{2n}(\RR) & \fraku_n
\\
\text{Case $\CC$} 
& \fraksp_{2m(p+q)}(\RR) 
& \fraku_m \oplus \fraku_m & \fraku_m 
& \fraku_{p,q} & \fraku_p \oplus \fraku_q
\\
\text{Case $\HH$} 
& \fraksp_{4mn}(\RR) 
& \fraku_{2m} & \frakusp_m & \frako^\ast_{2n} & \fraku_n
\end{array}
\end{gather}
We get the Capelli identities for symmetric pairs in a complete form
only for Case $\CC$ in the table above.
For Cases $\RR$ and $\HH$,
we only have explicit expressions of $\omega(X_d) \in \PD(V)$, 
and do not get $C_d \in U(\frakm)^H$ up to now.  

In Section~\ref{sec:C}
we prove the Capelli identities for Case $\CC$,
and in Sections~\ref{sec:R} and \ref{sec:H}
we give the equations up to $\PD(V)$ for Cases $\RR$ and $\HH$,
respectively.
In addition,
we show a formula for coefficients appearing 
in the Capelli identities in Appendix~\ref{sec:c^d_J}.

\section{Case $\CC$}
\label{sec:C}

Here we establish two kinds of the Capelli identities for symmetric pairs
for Case $\CC$ in Table  (\ref{table:see-saw-pairs}).
The first identity in Subsection~\ref{subsec:C-1} has a simple expression 
as differential operators in $\PD(V)$,
but their inverse images (Capelli elements) $X_d \in U(\frakg)^K$
are more complicated than the second one 
in Subsection~\ref{subsec:C-2}.
The second identity has a simple Capelli element $X_d \in U(\frakg)^K$,
for which it is easy to see the relation to generators of 
$S(\frakp)^K$.
These two types of the Capelli identities become equal 
when taking principal symbols.
Therefore they should be translated to each other by
$\CC$-linear combinations.

\subsection{Formulas for the Weil representation}

In this subsection, we give explicit formulas for the Weil representation.  
Let us recall our see-saw pair:
\begin{align*}
\begin{array}{cccccccc}
\frakg_0 &=& \fraku_m \oplus \fraku_m & & \fraku_{p,q} &=& \frakm_0 \\[-.2ex]
&&           \cup     & \text{\LARGE$\times$} & \cup \\[-.3ex]
\frakk_0 &=& \fraku_m & & \fraku_p \oplus \fraku_q &=& \frakh_0
\end{array}
\end{align*}
Here $\frakh_0 = \fraku_p \oplus \fraku_q$ 
is diagonally embedded into the Lie algebra $\frakm_0 = \fraku_{p,q}$ of the indefinite unitary group 
(note that $ \frakh_0 $ is the Lie algebra of a maximal compact subgroup of $ U(p, q) $); 
and $\frakg_0$ decomposes into $\frakk_0$ and $\frakp_0$ as follows:
\begin{gather*}
\frakg_0 = \frakk_0 \oplus \frakp_0,\qquad
\frakk_0 = \{ (A, -\T A) \in \frakg_0 \}, \qquad
\frakp_0 = \{ (A, \T A) \in \frakg_0 \}.
\end{gather*}
The real Lie algebra $\frakm_0$ is
embedded into $\fraksp_{2m(p+q)}(\RR)$ as follows.
\begin{gather*}
\begin{array}{cc@{}ccccccc}
\frakm_0 &=& \fraku_{p,q} & \hookrightarrow & \fraksp_{2m(p+q)}(\RR) \\[5pt]
&& A + \sqrt{-1} B & \mapsto &
\begin{pmatrix}
A^{\oplus m} & (-B I_{p,q})^{\oplus m} \\
(I_{p,q} B)^{\oplus m} & (I_{p,q} A I_{p,q})^{\oplus m}
\end{pmatrix}
\end{array}
\quad (A, B \in \Mat(p+q; \RR)),
\end{gather*}
where $I_{p,q}$ and $A^{\oplus m}$ is defined by 
\begin{align*}
I_{p,q} &=
\begin{pmatrix} 1_p & 0 \\ 0 & -1_q \end{pmatrix},
\qquad
A^{\oplus m} = 1_m \ast A =
\begin{smallpmatrix}
A & 0 & \cdots & 0 \\
0 & A &        & 0 \\[-.6ex]
\vdots & & \ddots & \vdots \\
0 & \cdots & 0 & A
\end{smallpmatrix}
\end{align*}
($m$ copies of $A$ on the diagonal).  
The real Lie algebra $\frakg_0$ is
embedded into $\fraksp_{2m(p+q)}(\RR)$ as follows.
{
\newcommand{\hatp}{\widehat{p}}
\newcommand{\hatq}{\widehat{q}}
\begin{multline*}
    \frakg_0 = \fraku_m \oplus \fraku_m \ni 
    (A + \sqrt{-1}B, A' + \sqrt{-1}B') \longmapsto 
    \\
    \begin{pmatrix}
      A \ast I_{p,\hatq} + A' \ast I_{\hatp,q} & \quad
      -B \ast I_{p,\hatq} - B' \ast I_{\hatp,q} \\
      B \ast I_{p,\hatq} + B' \ast I_{\hatp,q} & \quad
      A \ast I_{p,\hatq} + A' \ast I_{\hatp,q} 
    \end{pmatrix} 
    \in \fraksp_{2m(p+q)}(\RR) ,
\end{multline*}
where $A, B, A', B'$ are real matrices; and 
$I_{p,\hatq} , I_{\hatp,q}$ are defined by
\begin{align*}
I_{p,\hatq} = 
\begin{pmatrix} 1_p & 0 \\ 0 & 0_q \end{pmatrix}, 
\quad
I_{\hatp,q} = 
\begin{pmatrix} 0_p & 0 \\ 0 & 1_q \end{pmatrix}
\quad \in \Mat(p+q; \RR).  
\end{align*}
}
The notation $A \ast X$ means the Kronecker product
\begin{align*}
A \ast X = 
\begin{pmatrix}
a_{11} X & a_{12} X & \cdots & a_{1m} X \\
a_{21} X & a_{22} X & \cdots & a_{2m} X \\[-5pt]
\vdots  & \vdots   & \ddots & \vdots \\[-5pt]
a_{m1} X & a_{m2} X & \cdots & a_{mm} X
\end{pmatrix}.
\end{align*}

Let $V = \Mat(m, p; \CC) \oplus \Mat(m,q;\CC)$, 
and denote the Weil representation of $\fraksp_{2m(p+q)}(\CC)$ 
on the polynomial ring $\CC[V]$ by $\omega$.
Through the embeddings into $\fraksp_{2m(p+q)}(\CC)$,
the complexified Lie algebras $\frakg$ and $\frakm$
act on $\CC[V]$.
We denote these representations also by $\omega$.
Let $x_{s,i}$ ($1 \le s \le m, 1 \le i \le p$) 
and $y_{s,i}$ ($1 \le s \le m, 1 \le i \le q$) 
be the natural linear coordinate system of $V$,
and define the following matrices.
\begin{align*}
X &=  (x_{s,i})_{\substack{1 \le s \le m \\ 1 \le i \le p}} , 
\quad
\partial^X =
(\partial/\partial x_{s,i})_{\substack{1 \le s \le m \\ 1 \le i \le p}}
\quad
\in \Mat(m,p; \PD(V)),
\\
Y &=  (y_{s,i})_{\substack{1 \le s \le m \\ 1 \le i \le q}} ,
\quad
\partial^Y =
(\partial/\partial y_{s,i})_{\substack{1 \le s \le m \\ 1 \le i \le q}}
\quad
\in \Mat(m,q; \PD(V)),
\\
P &=  (X, \partial^Y) , 
\quad
Q =  (\partial^X, Y) 
\quad
\in \Mat(m,p+q; \PD(V)).
\end{align*}
The actions of basis elements of
$\frakg = \frakgl_m(\CC) \oplus \frakgl_m(\CC)$ is given by
\begin{align*}
\omega( (E_{s,t}, 0) ) &= 
{\textstyle\sum\limits_{i=1}^p x_{s,i} \partial^X_{t,i}} + \dfrac{p}{2} \delta_{s,t}, 
\quad 
\omega( (0, E_{s,t}) ) = 
{\textstyle\sum\limits_{i=1}^q y_{s,i} \partial^Y_{t,i}} + \dfrac{q}{2} \delta_{s,t}
\end{align*}
$(1 \le s,t \le m)$,
where $E_{s,t} \in \frakgl_m(\CC)$ denotes the matrix unit
with $1$ at the $(s,t)$-entry,
and $\delta_{s,t}$ denotes Kronecker's delta.
For this, see (4.5) of \bracketcite{MR1845714} for example.  
Note that $x_{s,i}$ in this article
corresponds to $x_{(p+q)(s-1)+i}$ in \bracketcite{MR1845714},
and $y_{s,i}$ in this article
corresponds to $x_{(p+q)(s-1)+p+i}$ in \bracketcite{MR1845714}.
The action of $\frakm = \frakgl_{p+q}(\CC)$ is given by
\begin{align*}
&
\omega\bigl(
  \begin{pmatrix} E_{i,j} & 0 \\ 0 & 0 \end{pmatrix}
\bigr) \! = \!\!
\textstyle
\sum\limits_{s=1}^m x_{s,i} \partial^X_{s,j} + \! \dfrac{m}{2}\delta_{i,j}, 
\quad 
\omega\bigl(
  \begin{pmatrix} 0 & 0 \\ 0 & E_{i,j} \end{pmatrix}
\bigr) \! = \!
\textstyle
- \!\! \sum\limits_{s=1}^m y_{s,j} \partial^Y_{s,i} - \! \dfrac{m}{2}\delta_{i,j}, 
\\
&
\omega\bigl(
  \begin{pmatrix} 0 & E_{i,j} \\ 0 & 0 \end{pmatrix}
\bigr) =
\textstyle
\sqrt{-1} \sum\limits_{s=1}^m x_{s,i} y_{s,j},
\qquad 
\omega\bigl(
  \begin{pmatrix} 0 & 0 \\ E_{j,i} & 0 \end{pmatrix}
\bigr) =
\textstyle
\sqrt{-1} \sum\limits_{s=1}^m \partial^X_{s,i} \partial^Y_{s,j},
\end{align*}
where $E_{i,j}$'s should be interpreted in the suitable sizes
(namely $p\times p$, $q\times q$, $p\times q$ or $q\times p$).
Note that we used the normalization 
$x_{s,i} \mapsto \sqrt{2} x_{s,i}$ and
$y_{s,i} \mapsto -\sqrt{-2} y_{s,i}$ comparing to
(4.5) of \bracketcite{MR1845714}.

It is convenient to express the formulas of the Weil representation 
in a matrix form.
We define matrices by arranging basis elements of
$\frakg = \frakgl_m(\CC) \oplus \frakgl_m(\CC)$ or
$\frakm = \frakgl_{p+q}(\CC)$.
Set
\begin{gather}
  \label{eq:caseC-def-of-EX-EY}
  \EE^X = \bigl( (E_{s,t}, 0) \bigr)_{1\le s,t \le m},
  \qquad
  \EE^Y = \bigl( (0, E_{s,t}) \bigr)_{1\le s,t \le m},
  \\
  \label{eq:caseC-def-of-B}
  \BB' = (E_{i,j})_{1 \le i,j \le p+q},
  \qquad
  \BB = 
  \begin{pmatrix} 1_p & 0 \\ 0 & -\sqrt{-1} \,1_q \end{pmatrix}
  \BB'
  \begin{pmatrix} 1_p & 0 \\ 0 & -\sqrt{-1} \,1_q \end{pmatrix},
  \\
  \nonumber
  ( \EE^X, \EE^Y \in \Mat(m; U(\frakg)), \quad
    \BB', \BB \in \Mat(p+q; U(\frakm)) ).
\end{gather}
Then we can write the formulas above in a matrix form:
\begin{gather}
  \label{eq:caseC-image-of-EX-EY}
  \omega(\EE^X) = X \T \partial^X + \frac{p}{2} 1_m,
  \qquad
  \omega(\T \EE^Y) = \partial^Y \T Y - \frac{q}{2} 1_m,
  \\
  \label{eq:caseC-image-of-EX+tEY}
  \omega(\EE^X + \T\EE^Y) = P \T Q + \frac{p-q}{2} 1_m,
  \\
  \label{eq:caseC-image-of-B}
  \omega(\BB) = \T P Q + \frac{n}{2} I_{p,q}.
\end{gather}
For example,
$\omega(\EE^X)$ is by definition the $m\times m$ matrix
whose $(s,t)$-entry is $\omega(\EE^X_{s,t})$,
and it is equal to the $(s,t)$-entry of the $m\times m$ matrix 
$X \T \partial^X + (p/2) 1_m$.

Let
$\comb{m}{d} = \{ S \subset \{1,2,\ldots,m\} ; \#S = d \}$,
and $A_{S,T}$ denotes the submatrix of an $m\times m$ matrix $A$
with its rows and columns chosen from 
$S, T \subset \comb{m}{d}$.
It is not so difficult to see that
\begin{align*}
  {\textstyle\sum\nolimits_{S \in \comb{m}{d}} \det( \EE^X + \T\EE^Y )_{S,S}}
  \in S(\frakp)^K
  \quad
  (d = 1, 2, \ldots, m)
\end{align*}
is a generating set of $S(\frakp)^K$,
where $\EE^X$ and $\EE^Y$ are considered as 
matrices with entries $E_{s,t}$ in $S(\frakg)$.

\subsection{Capelli identity for Case $\CC$ {\upshape (1)}}
\label{subsec:C-1}

Here we give the first form of the Capelli identities for Case $\CC$.
We prove the identities in Subsection~\ref{subsec:C-1-pf}.
This form of the identities has a simple expression 
as differential operators,
but their inverse images (Capelli elements) $X_d \in U(\frakg)^K$
are more complicated than the second one 
given in Subsection~\ref{subsec:C-2}.

Let us recall the picture of the Capelli identities
for Case $\CC$:
\begin{align}
\label{eq:CaseC-picture}
\begin{array}{ccccccc}
U(\frakg)^K & \xrightarrow{\quad\omega\quad}
& \PD(V)^{K\times H} &
\xleftarrow{\quad\omega\quad}
& U(\frakm)^H 
\\
|| &&&& ||\\
\makebox[0pt][c]{
  $U(\frakgl_m(\CC) \oplus \frakgl_m(\CC))^{GL_m(\CC)}$}
&&&&
\makebox[0pt][c]{$U(\frakgl_{p+q}(\CC))^{GL_p(\CC)\times GL_q(\CC)}$}
\end{array}
\end{align}
We first give the formula 
which corresponds to $U(\frakg)^K \to \PD(V)^{K\times H}$
in Proposition~\ref{prop:caseC-easy-left-half}.
We often write the elements of $S \in \comb{m}{d}$ by
$S(1), S(2), \ldots, S(d)$ 
or $s_1, s_2, \ldots, s_d$
in increasing order.
For two disjoint index sets 
$S' \in \comb{m}{d}$ and $S'' \in \comb{m}{m-d}$,
let $l(S',S'')$ denotes the inversion number 
of the concatenated sequence $(S', S'')$.

\begin{definition}
\label{defn:det-param}
We define the column-determinant with diagonal parameters 
$ u = ( u_1 , u_2 , \dots , u_l ) $ by
\begin{align*}
\det(A_{S',T'}; u )
& =
\sum\nolimits_{\sigma\in\Symm_l} \sgn(\sigma)
(A_{S'(\sigma(1)),T'(1)} + u_1 \delta_{S'(\sigma(1)),T'(1)})
\cdots
\\[-5pt]
& \hspace{.3\textwidth}
\cdots
(A_{S'(\sigma(d)),T'(d)} + u_d \delta_{S'(\sigma(d)),T'(d)})
\\[3pt]
&=
\det( A_{S'\!,T'} + 1_{S'\!,T'}
\left(\begin{smallmatrix}
u_1     &   0       &  \cdots  &   0  \\
0       &   u_2     &  \cdots  &   0  \\
\vdots  &&  \ddots  &  \vdots  \\
0       &   0       &  \cdots  &   u_d
\end{smallmatrix}\right)
),
\end{align*}
for an $m\times m$ matrix $A$
and $S', T' \in \comb{m}{l}$.  
Here $ 1_{S', T'} $ denotes the $ ( S' , T' ) $-submatrix of the identity matrix.
\end{definition}

\begin{proposition} 
\label{prop:caseC-easy-left-half}  
For $S, T \in \comb{m}{d}$, the invariant differential operator \linebreak
$\sum_{J\in\comb{p+q}{d}} \det P_{S,J} \det Q_{T,J} \in \PD(V)$ 
can be expressed as an image of
$U(\frakg) = U(\frakgl_m(\CC) \oplus \frakgl_m(\CC))$
under the Weil representation $\omega$
as follows.
\begin{align*}
  \lefteqn{
    \sum\nolimits_{J\in\comb{p+q}{d}} \det P_{S,J} \det Q_{T,J}}
  \\
  &=
  \sum_{l=0}^d 
  \sum_{\substack{
      S'\!,T' \\
      S\dblprime\!,T\dblprime}}
  (-1)^{\scriptscriptstyle l(S'\!,S\dblprime) + l(T'\!,T\dblprime)}
\det( \omega(\EE^X)_{S'\!,T'}; \alpha) \, 
\det( \omega(\T\EE^Y)_{S\dblprime\!,T\dblprime}; \beta),
\end{align*}
where the second summation is taken over 
$      S',T' \in \comb{m}{l}, \;  S\dblprime,T\dblprime \in \comb{m}{d-l} $ such that 
$ S' \amalg S\dblprime = S $ and $ T' \amalg T\dblprime = T $; 
and $ \alpha $ and $ \beta $ denote
\begin{align*}
\alpha &= \textstyle ( l-1-\frac{p}{2}, l-2-\frac{p}{2}, \ldots, -\frac{p}{2}) , \\
\beta &= \textstyle (-(d-l-1)+\frac{q}{2}, -(d-l-2)+\frac{q}{2}, \ldots, \frac{q}{2}).
\end{align*}
\end{proposition}
The proof is given in \S~\ref{subsubsec:C-1-pf-left-half}.

\begin{remark}
\label{rmk:cauchy-binet}
In general,
for $n\times n$ matrices $A, B$ with commutative entries,
we have 
\begin{gather*}
\textstyle
  \det(A+B) =
  \sum\limits_{l=0}^n
  \sum\limits_{\substack{
      S'\!,T'\!, S\dblprime\!,T\dblprime}}
  (-1)^{\scriptscriptstyle l(S'\!,S\dblprime) + l(T'\!,T\dblprime)} 
  \det A_{S'\!,T'} \det B_{S\dblprime\!,T\dblprime},
\end{gather*}
where the second summation is taken over 
$      S',T' \in \comb{m}{l}, \;  S\dblprime,T\dblprime \in \comb{m}{d-l} $ such that 
$ S' \amalg S\dblprime = T' \amalg T\dblprime = \{1,\ldots,n\} $.  
Hence the right-hand side of 
Proposition~\ref{prop:caseC-easy-left-half}
is equal to $\det( \omega(\EE^X + \T\EE^Y) )_{S,T}$,
when taking the principal symbols in 
the symmetric algebra $S(\frakg)$.  
Therefore Proposition~\ref{prop:caseC-easy-left-half}
can be regarded as a non-commutative version of 
\begin{gather*}
  \sum\nolimits_{J\in\comb{p+q}{d}} \det A_{S,J} \det B_{T,J} =
  \det(A \T B)_{S,T}
  \qquad
  (\text{Cauchy-Binet}).
\end{gather*}
\end{remark}

We next give the formula which corresponds to
$\PD(V)^{K\times H} \leftarrow U(\frakm)^H$.  
To state it, we need a variant of the symmetrized determinant.

\begin{definition}
\label{defn:Dets}
For an $n \times n$ matrix $A$,
we define 
\begin{align*}
\Det(A) =
\frac{1}{n!} \sum\nolimits_{\sigma,\tau\in\Symm_n} \sgn(\sigma) \sgn(\tau)
A_{\sigma(1),\tau(1)} A_{\sigma(2),\tau(2)} \cdots 
A_{\sigma(n),\tau(n)},
\end{align*}
which is called the \emph{symmetrized determinant}.  
Also we define the symmetrized determinant with diagonal parameters 
$ u = ( u_1, u_2, \dots , u_n ) $ by
\begin{align*}
\Det(A; u) 
&=
\frac{1}{n!} \sum\nolimits_{\sigma,\tau\in\Symm_n} \sgn(\sigma) \sgn(\tau)
(A_{\sigma(1),\tau(1)} + u_1 \delta_{\sigma(1),\tau(1)}) \cdots
\\ & \hspace{30ex}
\cdots
(A_{\sigma(n),\tau(n)} + u_n \delta_{\sigma(n),\tau(n)}).
\end{align*}
Next, we define a minor of the symmetrized determinant with uneven diagonal shift.
Let $p, q$ be non-negative integers and $n = p+q$.
For an $n\times n$ matrix $B$, $I, J \in \comb{n}{d}$ and diagonal parameters 
$ u = ( u_1, u_2, \dots , u_d ) $, 
we define
\begin{align*} 
  \Det_{p,q}(B_{I,J}; u) &=
\frac{1}{d!}
\sum_{\sigma, \tau\in\Symm_d} \!\! \sgn(\sigma) \sgn(\tau)
(B_{i_{\sigma(1)},j_{\tau(1)}} \! - u_1 \veps_{i_{\sigma(1)},j_{\tau(1)}})
\cdots \\[-.5ex]
& \hspace*{.35\textwidth}
\cdots (B_{i_{\sigma(d)},j_{\tau(d)}} \! - u_d \veps_{i_{\sigma(d)},j_{\tau(d)}}),
\end{align*}
where 
$ I = \{ i_1, i_2, \dots , i_d \} , \; 
  J = \{ j_1, j_2, \dots , j_d \} $ and 
$\veps_{i,j}$ is a variant of Kronecker's delta defined by
\begin{align*}
\veps = 
-I_{p,q} =
\begin{pmatrix} -1_p & 0 \\ 0 & 1_q \end{pmatrix}.
\end{align*}
\end{definition}

\begin{proposition} 
\label{prop:caseC-easy-right-half}  
\quad
For $I, J \in \comb{n}{d}$ and $ n = p + q $,
the differential operator \linebreak
$\sum_{S\in\comb{m}{d}} \det P_{S,I} \det Q_{S,J} \in \PD(V)$
can be expressed as an image of $U(\frakm)$
under the Weil representation $\omega$ as follows.
\begin{align*}
{\textstyle\sum\nolimits_{S\in\comb{m}{d}} \det P_{S,I} \det Q_{S,J}}
&=
\Det_{p,q}( \bigl( \omega(\BB) - \frac{n}{2} I_{p,q} \bigr)_{\! I,J};
d-1, d-2, \ldots, 0).
\end{align*}
\end{proposition}
The proof of this proposition is given in \S~\ref{subsubsec:C-1-pf-right-half}.

For $d \ge 1$,
we define
$X_d \in U(\frakg)^K =
U(\frakgl_m(\CC) \oplus\frakgl_m(\CC))^{GL_m(\CC)}$
by
\begin{align*}
  X_d &= \textstyle
  \sum\limits_{S\in\comb{m}{d}}
  \sum\limits_{l=0}^d
  \sum\limits_{\substack{
  S'\!,S\dblprime \\ T'\!,T\dblprime
  }}
  (-1)^{\scriptscriptstyle l(S'\!,S\dblprime)+l(T'\!,T\dblprime)}
  \det( \EE^X_{S'\!,T'}; \alpha) 
  \det( (\T\EE^Y)_{S\dblprime\!,T\dblprime}; \beta) , 
\end{align*}
where
$ \alpha = (l-1-{p}/{2}, l-2-{p}/{2}, \ldots, -{p}/{2}) , \; 
  \beta  = (-(d-l-1)+{q}/{2}, -(d-l-2)+{q}/{2}, \ldots, {q}/{2}) $, 
and the third summation is taken over 
$ S', T' \in \comb{m}{l} $ and $ S\dblprime, T\dblprime \in \comb{m}{d-l} $ which satisfy 
$ S = S' \amalg S\dblprime = T' \amalg T\dblprime $.  

Similarly we define 
$C_d \in U(\frakm)^H = 
U(\frakgl_{p+q}(\CC))^{GL_p(\CC)\times GL_q(\CC)}$ by 
\begin{align*}
  C_d &=
\textstyle
  \sum_{J\in\comb{p+q}{d}}
  \Det_{p,q}( \bigl( \BB - \frac{n}{2} I_{p,q} \bigr)_{J,J};
    d-1, d-2, \ldots, 0) .
\end{align*}
%

Finally we combine 
Propositions~\ref{prop:caseC-easy-left-half} and
\ref{prop:caseC-easy-right-half},
and obtain the Capelli identity for Case $\CC$.

\begin{theorem}[Capelli identity for Case {$\CC$} (1)]
\label{thm:caseC-easy}
Under the above notation, we have the Capelli identity
\begin{align*}
\omega(X_d) 
= 
  \sum_{\substack{S\in\comb{m}{d}, J\in\comb{p+q}{d}}}
  \det P_{S,J} \det Q_{S,J}
=
\omega(C_d).
\end{align*}
We call $X_d$ and $C_d$ \emph{Capelli elements}.
\end{theorem}

\begin{remark}
\label{rmk:to-thm:caseC-easy}
As already noted in Remark~\ref{rmk:cauchy-binet}, 
we can prove that the principal symbol of $X_d$ is 
\begin{align*}
\textstyle
  \sum_{S\in\comb{m}{d}} \det(\EE^X + \T\EE^Y)_{S,S}
  \quad\in
  S(\frakp)^K \simeq S(\Mat(m;\CC))^{GL_m(\CC)},
\end{align*}
and these invariant elements generate $S(\frakp)^K$.
Thus the left-hand side of the theorem can be considered as
the image under the Weil representation
of invariant elements corresponding to the generators of
$S(\frakp)^K$.
These invariant differential operators
are expressed by the elements of
$U(\frakm)^H = U(\frakgl_{p+q})^{GL_p(\CC)\times GL_q(\CC)}$
on the right-hand side.
\end{remark}

\subsection{Capelli identity for Case $\CC$ {\upshape (2)}}
\label{subsec:C-2}

Here we give the second form of the Capelli identities for Case $\CC$. 
This form of the identities has simpler Capelli elements
and it is easy to see the relation to  generators of $S(\frakp)^K$.
Let us recall the picture~(\ref{eq:CaseC-picture}) for Case $\CC$.
We first give the formula which corresponds to
$U(\frakg)^K \to \PD(V)^{K\times H}$.

\begin{proposition}
\label{prop:caseC-not-easy-left-half}
For $d \ge 1$, we define 
$X_d \in U(\frakg)^K $ by
\begin{align*}
  X_d = \textstyle
  \sum_{S\in\comb{m}{d}} 
  \Det(( \EE^X+\T\EE^Y - \frac{p-q}{2}1_m )_{S,S}),
\end{align*}
where $\EE^X$ and $\EE^Y$ are given in 
Equation~{\upshape(\ref{eq:caseC-def-of-EX-EY})}.
Then its image under the Weil representation $\omega$ is given by
\begin{align*}
\omega(X_d)  &= 
\textstyle
  \sum\limits_{l=0}^d \dfrac{(m-l)!}{d!(m-d)!}
  \sum_{J\in\comb{p+q}{l}} c^d_J
  \sum_{S\in\comb{m}{l}} \det P_{S,J} \det Q_{S,J}.
\end{align*}
In the above formula, 
$c^d_J$ is an integer defined by
\begin{align}
  \label{eq:def-of-c^d_J}
  & 
  c^d_J = c^d_{\alpha,\beta}
  = \textstyle
  \sum_{k\in\ZZ} \binom{\alpha+\beta}{k} (\beta-k)^d (-1)^k
\\
&
\label{eq:def-of-c^d_J-alpha-beta}
\alpha = \#\{ i \;;\; 1 \le J(i) \le p \}, \quad
\beta = \#\{ i \;;\; p+1 \le J(i) \le p+q \}, 
\end{align}
for a non-negative integer $d$ and $J \in \comb{p+q}{l}$.  
Note that 
$ d = \alpha + \beta $.  
\end{proposition}
The proof is given in \S~\ref{subsec:C-2-pf}.


We next give the formula corresponding to
$\PD(V)^{K\times H} \leftarrow U(\frakm)^H$.
\begin{proposition}
\label{prop:caseC-not-easy-right-half}
For $d \ge 1$, 
define $C_d \in U(\frakm)^H $ by
\begin{align*}
\textstyle
  C_d = 
  \sum\limits_{l=0}^d \dfrac{(m-l)!}{d!(m-d)!}
  \sum\limits_{J\in\comb{p+q}{l}} c^d_J
  \Det_{p,q}( (\BB-\frac{m}{2}I_{p,q})_{J,J}; l-1, l-2, \ldots, 0),
\end{align*}
where $c^d_J$ is defined 
in Equation \eqref{eq:def-of-c^d_J} 
and $\BB \in \Mat(p+q;U(\frakm))$ in Equation~\eqref{eq:caseC-def-of-B}.
Then its image under the Weil representation $\omega$
is given by
  \begin{align*}
  \textstyle
    \omega(C_d) = 
      \sum\limits_{l=0}^d \dfrac{(m-l)!}{d!(m-d)!}
      \sum_{J\in\comb{p+q}{l}} c^d_J
      \sum_{S\in\comb{m}{l}} \det P_{S,J} \det Q_{S,J}.
  \end{align*}
\end{proposition}
\begin{proof}
The formula is a linear combination over the elements in Proposition~\ref{prop:caseC-easy-right-half}.
\end{proof}

Thus we get the following theorem.
\begin{theorem}[Capelli identity for Case {$\CC$} (2)]
  \label{thm:caseC-not-easy}
  For $d \ge 1$, 
  we define $X_d \in U(\frakg)^K =
  U(\frakgl_m(\CC) \oplus\frakgl_m(\CC))^{GL_m(\CC)}$
  and $C_d \in U(\frakm)^H = 
  U(\frakgl_{p+q}(\CC))^{GL_p(\CC)\times GL_q(\CC)}$
  as in Propositions~\ref{prop:caseC-not-easy-left-half}
  and \ref{prop:caseC-not-easy-right-half}, respectively.
  Then we have
  \begin{align*}
    \omega(X_d) = 
    \sum_{l=0}^d \dfrac{(m-l)!}{d!(m-d)!}
    \sum_{J\in\comb{p+q}{l}} c^d_J
    \sum_{S\in\comb{m}{l}} \det P_{S,J} \det Q_{S,J}
    = 
    \omega(C_d).
  \end{align*}
\end{theorem}

\subsection{Proof of the Capelli identity for Case $\CC$ {\upshape(1)}}
\label{subsec:C-1-pf}

Here we give the proof of Propositions~\ref{prop:caseC-easy-left-half}
and \ref{prop:caseC-easy-right-half} in \S~\ref{subsec:C-1}.

\subsubsection{Proof of Proposition~\ref{prop:caseC-easy-left-half}}
\label{subsubsec:C-1-pf-left-half}

Let $e_s$ and $e^\ast_s$ be the elements in standard basis of 
$\CC^m$ and its dual $(\CC^m)^\ast$, respectively.
We consider the exterior algebra 
$\bigwedge (\CC^m \oplus (\CC^m)^\ast)$
and the ring $\PD(V)$ of differential operators
with polynomial coefficients on $V$.
We define several elements in the algebra
$\bigwedge (\CC^m \oplus (\CC^m)^\ast) \otimes_\CC \PD(V)$:
\begin{gather*}
  \textstyle
  \alpha_j = \sum\limits_{s=1}^m e_s P_{s,j},
  \quad
  \beta_j = \sum\limits_{s=1}^m e^\ast_s Q_{s,j}
  \quad (1 \le j \le p+q) ; 
  \quad
  \tau = \sum\limits_{s=1}^m e_s e^\ast_s,
  \\
  \textstyle
  \Xi_X =
  \sum\limits_{j=1}^p \alpha_j \beta_j =
  \sum\limits_{s,t=1}^m e_s e^\ast_t (X {\cdot} \T\partial^X)_{(s,t)},
  \\
  \textstyle
  \Xi_Y =
  \sum\limits_{j=p+1}^{p+q} \alpha_j \beta_j =
  \sum\limits_{s,t=1}^m e_s e^\ast_t (\partial^Y \! {\cdot} \T{Y})_{(s,t)}.
\end{gather*}
Here $(X {\cdot} \T\partial_X)_{(s,t)}$ denotes the $(s,t)$-entry of 
the matrix $X {\cdot} \T\partial_X$, 
and we omit $ \wedge $ for the multiplication in the exterior algebra.  
For an index set $J$,
we put $\alpha_J = \alpha_{j_1} \cdots \alpha_{j_d}$.
We also define $\beta_J$, $e_S$ or $e^\ast_T$ in the same way.
Then $\alpha_J$ and $\beta_J$ are
written in terms of column-determinants as follows:
\begin{equation}
  \label{eq:alpha-beta-causes-det}
  \alpha_J = \textstyle\sum_{S\in\comb{m}{d}} e_S \det P_{S,J},
  \qquad
  \beta_J = \textstyle\sum_{T\in\comb{m}{d}} e^\ast_T \det Q_{T,J}.
\end{equation}
Recall the matrix $\veps = - I_{p,q} \in \Mat(p+q ; \CC)$.
We have the following relations for the elements defined above.

\begin{lemma}
\begin{thmenumerate}
\item
The element $\tau$ is central in the algebra
$\bigwedge (\CC^m \oplus (\CC^m)^\ast) \otimes_{\CC} \PD(V)$. 
\item
We have the following commutation relations:
\begin{align}
\label{eq:comm-rel-of-alpha-alpha}
& \alpha_i \alpha_j + \alpha_j \alpha_i  = 0,
\qquad \beta_i \beta_j + \beta_j \beta_i  = 0
\quad (1 \le i,j \le p+q),
\\
\label{eq:comm-rel-of-P-Q}
& [P_{s,i}, Q_{t,j}] = \veps_{i,j} \delta_{s,t}
\quad (1 \le i,j \le p+q,\; 1 \le s,t \le m),
\\
\label{eq:comm-rel-of-alpha-beta}
& \alpha_i \beta_j + \beta_j \alpha_i = \veps_{i,j}\tau
\quad (1 \le i,j \le p+q),
\\ 
\label{eq:comm-rel-of-XiX-XiY}
& [\Xi_X, \Xi_Y] = 0,
\\
\label{eq:comm-rel-of-alpha-XiX}
& \alpha_j \Xi_X = (\Xi_X+\tau) \alpha_j \quad (1\le j \le p),
\\
\label{eq:comm-rel-of-alpha-XiY}
& \alpha_j \Xi_Y = (\Xi_Y-\tau) \alpha_j \quad (p+1 \le j \le p+q).
\end{align}
\end{thmenumerate}
\end{lemma}
\begin{proof}
(1)
$2m$ elements $e_s$ ($1 \le s \le m$) 
and $e^\ast_s$ ($1 \le s \le m$) are anti-commutative,
and $\tau$ is a sum of $e_s e^\ast_s$,
which is commutative with $e_t$ and $e^\ast_t$.
Therefore $\tau$ is central in
$\bigwedge (\CC^m \oplus (\CC^m)^\ast) \otimes_{\CC} \PD(V)$.

(2)-Equation~(\ref{eq:comm-rel-of-alpha-alpha}):
The entries of $P = (X, \partial^Y)$ commute with each other.
Hence $\alpha_i$ and $\alpha_j$ are anti-commutative.
Similarly $\beta_i$ and $\beta_j$ are also anti-commutative.

(2)-Equation~(\ref{eq:comm-rel-of-P-Q}):
$[P_{s,i}, Q_{t,j}]$ is nonzero only if $s=t$ and $i=j$.
If $1 \le i \le p$, then 
$[P_{s,i}, Q_{s,i}] = [x_{s,i}, \partial^X_{s,i}] = -1$.
If $p+1 \le i \le p+q$, then 
$[P_{s,i}, Q_{s,i}] = [\partial^Y_{s,i-p}, y_{s,i-p}] = 1$.
Hence we have 
$[P_{s,i}, Q_{t,j}] = \veps_{i,j} \delta_{s,t}$.

(2)-Equation~(\ref{eq:comm-rel-of-alpha-beta}):
We use Equation~(\ref{eq:comm-rel-of-P-Q}), and we get
\begin{align*}
\textstyle
\alpha_i \beta_j + \beta_j \alpha_i 
=
\sum_{s,t=1}^m e_s e^\ast_t [P_{s,i}, Q_{t,j}]
=
\sum_{s,t=1}^m e_s e^\ast_t \veps_{i,j} \delta_{s,t}
= \veps_{i,j} \tau.
\end{align*}

(2)-Equation~(\ref{eq:comm-rel-of-XiX-XiY}):
By Equations~(\ref{eq:comm-rel-of-alpha-alpha})
and (\ref{eq:comm-rel-of-alpha-beta}),
$\alpha_j$ and $\beta_j$ in
$\Xi_X = \sum_{j=1}^p \alpha_j \beta_j$ 
anti-commute with
$\alpha_j$ and $\beta_j$ in
$\Xi_Y = \sum_{j=p+1}^{p+q} \alpha_j \beta_j$.
Hence $\Xi_X$ commutes with $\Xi_Y$.

(2)-Equation~(\ref{eq:comm-rel-of-alpha-XiX}):
For $1 \le j \le p$, we have
\begin{align}
\nonumber
\alpha_j \Xi_X 
&=
\textstyle
\sum_{i=1}^p \alpha_j \alpha_i \beta_i
=
-\sum_{i=1}^p \alpha_i \alpha_j \beta_i \\
&=
\textstyle
-\sum_{i=1}^p \alpha_i (-\beta_i \alpha_j + \veps_{i,j} \tau)
=
\Xi_X \alpha_j + \alpha_j \tau,
\label{eq:proof:alpha-XiX}
\end{align}
where we used Equations~(\ref{eq:comm-rel-of-alpha-alpha})
and (\ref{eq:comm-rel-of-alpha-beta}).
Equation~(\ref{eq:comm-rel-of-alpha-XiY}) is proved 
similarly to Equation~(\ref{eq:comm-rel-of-alpha-XiX}).
\end{proof}
We have the following relations of 
$\Xi_X$ and $\Xi_Y$ with the symmetrized determinants.
\begin{lemma}
\label{lem:caseC-Xi-to-Det}
For indeterminate $ z $, set $\Xi_X(z) = \Xi_X + z\tau$ 
and similarly $\Xi_Y(z) = \Xi_Y + z\tau$.
Then, for the variables $ u = (u_1, u_2, \dots, u_d) $, we have
\begin{align*}
  \Xi_X(u_1) \Xi_X(u_2) \cdots \Xi_X(u_d)
  &=
  d! (-1)^{\frac{d(d-1)}{2}} \!\! \textstyle\sum\limits_{S,T\in\comb{m}{d}} e_S e^\ast_T
  \Det( (X\T\partial^X)_{S,T}; u),
  \\
  \Xi_Y(u_1) \Xi_Y(u_2) \cdots \Xi_Y(u_d)
  &=
  d! (-1)^{\frac{d(d-1)}{2}} \!\! \textstyle\sum\limits_{S,T\in\comb{m}{d}} e_S e^\ast_T
  \Det( (\partial^Y \T Y)_{S,T}; u ).
\end{align*}
\end{lemma}
\begin{proof}
Let $\Xi = \sum_{s,t=1}^m e_s e^\ast_t A_{s,t}$
be
$\Xi_X = \sum_{s,t=1}^m e_s e^\ast_t (X \T\partial^X)_{s,t}$ or
$\Xi_Y = \sum_{s,t=1}^m e_s e^\ast_t (\partial^Y \T Y)_{s,t}$.
Then we have
\begin{align*}
\textstyle
\Xi(u) = \Xi + u\tau
=
\sum\limits_{s,t=1}^m e_s e^\ast_t A_{s,t} + u \sum\limits_{s=1}^m e_s e^\ast_s
=
\sum\limits_{s,t=1}^m e_s e^\ast_t (A_{s,t} + u \delta_{s,t}).
\end{align*}
Hence we have
\begin{align*}
  & \Xi(u_1) \cdot \Xi(u_2) \cdots \Xi(u_d) \\
  &=
  \textstyle
  \sum\limits_{\substack{1 \le s_1,\ldots,s_d \le m \\ 1 \le t_1,\ldots,t_d \le m }}
  e_{s_1} e^\ast_{t_1} \cdots e_{s_d} e^\ast_{t_d}
  (A_{s_1,t_1} + u \delta_{s_1,t_1}) \cdots
  (A_{s_d,t_d} + u \delta_{s_d,t_d})
  \\ &=
  \textstyle
  \sum_{S,T\in\comb{m}{d}} \sum_{\sigma,\tau\in\Symm_d}
  \sgn(\sigma) \sgn(\tau) (-1)^{d(d-1)/2} 
  e_S e^\ast_T
  \\[-3pt] & \hspace{10ex}
  \times
  (A_{s_{\sigma(1)},t_{\tau(1)}} +
    u_1 \delta_{s_{\sigma(1)},t_{\tau(1)}} )
  \cdots
  (A_{s_{\sigma(d)},t_{\tau(d)}} +
    u_d \delta_{s_{\sigma(d)},t_{\tau(d)}} )
  \\[3pt] &=  
  \textstyle
  d! (-1)^{d(d-1)/2} \sum_{S,T\in\comb{m}{d}}
  e_S e^\ast_T \Det(A_{S,T}; u_1, u_2, \ldots, u_d).
\end{align*}
\end{proof}

\begin{proposition} 
\label{prop:caseC-easy-left-half-Det}
For $S,T \in \comb{m}{d}$, we have
\begin{align*}
  \lefteqn{
\textstyle
    \sum_{J\in\comb{p+q}{d}} \det P_{S,J} \det Q_{T,J}}
  \\
  &=
  \textstyle
  \sum\limits_{l=0}^d 
  \sum\limits_{\substack{
      S'\!,T' \\
      S\dblprime\!,T\dblprime}}
  (-1)^{\scriptscriptstyle l(S'\!,S\dblprime) + l(T'\!,T\dblprime)}
  \Det( (X \T \partial_X)_{S'\!,T'}; l-1, l-2, \ldots, 0)
  \\[-5pt]
  &\hspace*{.3\textwidth} \times
  \Det( (\partial_Y \T Y)_{S\dblprime\!,T\dblprime}; -(d-l-1),\ldots, -1, 0).
\end{align*}
where the second summation is taken over 
$      S',T' \in \comb{m}{l}, \;  S\dblprime,T\dblprime \in \comb{m}{d-l} $ such that 
$ S' \amalg S\dblprime = S $ and $ T' \amalg T\dblprime = T $.
\end{proposition}
\begin{proof}
We show the equality by computing 
$\sum_{J\in\comb{p+q}{d}} \alpha_J \beta_J$
in two different ways.
First, we observe
\begin{align*}
\textstyle
\sum_{J\in\comb{p+q}{d}} \alpha_J \beta_J = 
\sum_{J\in\comb{p+q}{d}} \sum_{S,T\in\comb{m}{d}}
e_S e^\ast_T \det P_{S,J} \det Q_{T,J},
\end{align*}
thanks to Equation~(\ref{eq:alpha-beta-causes-det}).
The coefficient of $e_S e^\ast_T$ in this expression is
equal to the left-hand side of the desired formula.

Second, we compute
$\sum_{J\in\comb{p+q}{d}} \alpha_J \beta_J$ as follows:
\begin{align}
\nonumber
\textstyle
\sum_{J\in\comb{p+q}{d}} \alpha_J \beta_J
&=
\nonumber
\textstyle
\sum\limits_{l=0}^d 
\sum_{J\in\comb{p}{l}} \sum_{K\in p+\comb{q}{d-l}}
\alpha_J \alpha_K \beta_J \beta_K
\\ &=
\label{eq:proof:sum-alphaJ-betaJ}
\textstyle
\sum\limits_{l=0}^d 
\sum_{J\in\comb{p}{l}} \sum_{K\in p+\comb{q}{d-l}}
(-1)^{l(d-l)} \alpha_J \beta_J \alpha_K \beta_K,
\end{align}
where
$p + \comb{q}{d-l} =
\{ \{p+k_1, p+k_2, \ldots, p+k_{d-l} \} ;
  K \in \comb{q}{d-l} \}$,
and we used the anti-commutativity of 
factors $\alpha_k$ of $\alpha_K$ and $\beta_j$ of $\beta_J$.
We compute the factor
$\sum_{J\in\comb{p}{l}} \alpha_J \beta_J$
in Equation~(\ref{eq:proof:sum-alphaJ-betaJ}) as follows:
\begin{align}
\nonumber
\textstyle
\sum_{J \in \comb{p}{l}} \alpha_J \beta_J
&= 
\frac{1}{l!} 
\textstyle
\sum\limits_{j_1,\ldots,j_l=1}^p
\alpha_{j_1} \alpha_{j_2} \cdots ( \alpha_{j_l} \beta_{j_l} ) 
\beta_{j_1} \beta_{j_2} \cdots \beta_{j_{l-1}}
\cdot (-1)^{l-1}
\\ &=
\nonumber
\frac{(-1)^{l-1}}{l!}
\textstyle
\sum\limits_{j_1,\ldots,j_{l-1}=1}^p
\alpha_{j_1} \alpha_{j_2} \cdots \alpha_{j_{l-1}} 
\Xi_X 
\beta_{j_1} \beta_{j_2} \cdots \beta_{j_{l-1}}
\\ &=
\label{eq:proof:sum-alphaJ-betaJ-2}
\frac{(-1)^{l-1}}{l!}
\textstyle
\sum\limits_{J' \in \{1,2,\ldots,p\}^{l-1}}
\Xi_X(l-1) \cdot \alpha_{J'} \beta_{J'},
\end{align}
where we used Equation~(\ref{eq:comm-rel-of-alpha-XiX})
in order to move $\Xi_X$ to the left.
By repeating this operation, we have
\begin{align*}
(\ref{eq:proof:sum-alphaJ-betaJ-2}) =
\frac{(-1)^{l(l-1)/2}}{l!}
\Xi_X(l-1) \cdot \Xi_X(l-2) \cdots \Xi_X(0).
\end{align*}
Similarly, for $ r = d - l $, we have
\begin{align*}
\textstyle
  \sum\limits_{K \in p+\comb{q}{r}} \!\!\! \alpha_K \beta_K 
  =
  \dfrac{(-1)^{r(r-1)/2}}{r!}
  \Xi_Y(-r+1) \Xi_Y(-r+2) \cdots \Xi_Y(0),
\end{align*}
where we used Equation~(\ref{eq:comm-rel-of-alpha-XiY}).
Thus Expression~(\ref{eq:proof:sum-alphaJ-betaJ})
is computed as follows:
\begin{align}
\nonumber
(\ref{eq:proof:sum-alphaJ-betaJ}) 
&=
\textstyle
\sum\limits_{l=0}^d 
(-1)^{l(d-l)} \cdot
\dfrac{(-1)^{l(l-1)/2}}{l!}
\Xi_X(l-1) \cdot \Xi_X(l-2) \cdots \Xi_X(0)
\\ & \quad
\nonumber
\times
\frac{(-1)^{(d-l)(d-l-1)/2}}{(d-l)!}
\Xi_Y(-d+l+1) \cdot \Xi_Y(-d+l+2) \cdots \Xi_Y(0)
\\ &=
\nonumber
\textstyle
\sum\limits_{l=0}^d 
(-1)^{l(d-l)} 
\sum\limits_{S'\!,T'\in\comb{m}{l}} \!\!
e_{S'} e^\ast_{T'} 
\Det( (X \T\partial^X)_{S'\!,T'}; l-1, \ldots, 0 )
\\ & \quad
\label{eq:proof:sum-alphaJ-betaJ-3}
\times
\textstyle
\sum\limits_{S\dblprime\!,T\dblprime\in\comb{m}{d-l}} \!\!
e_{S\dblprime} e^\ast_{T\dblprime} 
\Det( (\partial^Y \T Y)_{S\dblprime\!,T\dblprime}; -d+l+1, \ldots, 0 ),
\end{align}
where we used Lemma~\ref{lem:caseC-Xi-to-Det}.
Now we look at the coefficient of $e_S e^\ast_T$
in Equation~(\ref{eq:proof:sum-alphaJ-betaJ-3}).
A summand of (\ref{eq:proof:sum-alphaJ-betaJ-3})
has $e_S e^\ast_T$ only if 
$S' \amalg S\dblprime = S$ and $T' \amalg T\dblprime = T$.
Even in that case,
we have to translate
$e_{S'} e^\ast_{T'} e_{S\dblprime} e^\ast_{T\dblprime}$ to
$e_{S} e^\ast_{T}$
in order to determine the sign of the summand.
First, $(-1)^{d(d-l)}$ occurs by
moving $e_{S\dblprime}$ to the left of $e^\ast_{T'}$.
Second, $(-1)^{\scriptscriptstyle l(S'\!,S\dblprime)+l(T'\!,T\dblprime)}$ occurs by sorting
$e_{S'} e_{S\dblprime}$ and $e^\ast_{T'} e^\ast_{T\dblprime}$ to
$e_S$ and $e^\ast_T$, respectively.
%
Thus it turns out that the coefficient of
$e_S e^\ast_T$ in Equation~(\ref{eq:proof:sum-alphaJ-betaJ-3})
is equal to the right-hand side of the desired formula
of the proposition.
\end{proof}

\begin{lemma} 
\label{lem:Det-to-det}
Let $E_{i,j} \; (1 \le i,j \le n) $ be the matrix units, and
$\EE = (E_{i,j})_{1 \le i,j \le n} \in \Mat(n;U(\frakgl_n(\CC)))$.
Then the following relation 
between column-determinants and symmetrized determinants
holds in $U(\frakgl_n(\CC))$.
\begin{align*}
\det(\EE_{I,J}; u-1, & u-2, \ldots, u-d)
=
\Det(\EE_{I,J}; u-1, u-2, \ldots, u-d)
\\ & 
=
\Det((\T\EE)_{J,I}; u-d, u-d+1, \ldots, u-1)
\\ &
=
\det((\T\EE)_{J,I}; u-d, u-d+1, \ldots, u-1).
\end{align*}
\end{lemma}
\begin{proof}
The proof is by exterior calculus, and we omit it.
\end{proof}

\begin{proof}[Proof of Proposition~\ref{prop:caseC-easy-left-half}]
By Equation~(\ref{eq:caseC-def-of-EX-EY}), we have
\begin{align}
\label{eq:proof:XtdX-and-dYtY-by-EE}
X \T\partial^X = \omega(\EE^X) - \frac{p}{2}\,1_m,
&&
\partial^Y \T Y = \omega(\T\EE^Y) + \frac{q}{2}\,1_m,
\end{align}
and hence entries of $X \T\partial^X$ and $\partial^Y \T Y$
have the same commutation relations as those of $\EE^X$ and $\T \EE^Y$,
respectively.
Therefore symmetrized determinants on the right-hand side
of Proposition \ref{prop:caseC-easy-left-half-Det}
can be changed to column-determinants
thanks to Lemma~\ref{lem:Det-to-det}.
To the resulting formula,
we again apply Equation~(\ref{eq:proof:XtdX-and-dYtY-by-EE}),
and we have proved Proposition~\ref{prop:caseC-easy-left-half}.
\end{proof}

\subsubsection{Proof of Proposition~\ref{prop:caseC-easy-right-half}}
\label{subsubsec:C-1-pf-right-half}

Let $f_i$ and $f^\ast_i$ be the elements in the standard basis 
of $\CC^{p+q}$ and $(\CC^{p+q})^\ast$ respectively.
As before, we define some special elements in the algebra
$\bigwedge (\CC^{p+q} \oplus (\CC^{p+q})^\ast) \otimes_\CC \PD(V)$.
\begin{gather*}
\textstyle
  \eta_s = \sum_{i=1}^{p+q} f_i P_{s,i},
  \qquad
  \zeta_s = \sum_{i=1}^{p+q} f^\ast_i Q_{s,i}
  \quad (1 \le s \le m),
  \\
\textstyle
  \Lambda = \sum_{i,j=1}^{p+q} f_i f^\ast_j (\T P Q)_{(i,j)}
  = \sum_{s=1}^m \eta_s \zeta_s,
  \qquad
  \sigma = \sum_{i=1}^{p+q} \veps_{i,i} f_i f^\ast_i.
\end{gather*}
In this case, we obtain the row-determinants of $ P $ or $ Q $ 
when we make the products of $\eta_s$'s or $\zeta_s$'s.  
But the entries of $P$ or $Q$ are commutative with each other, 
so in fact, there is no difference between row and column-determinants.
Thus we have
\begin{gather*}
\textstyle
  \eta_{s_1} \eta_{s_2} \cdots \eta_{s_d}
  = \sum\limits_{I\in\comb{p+q}{d}} f_I \det P_{S,I},
\quad
  \zeta_{s_1} \zeta_{s_2} \cdots \zeta_{s_d}
  = \sum\limits_{I\in\comb{p+q}{d}} f^\ast_I \det Q_{S,I}.
\end{gather*}

\begin{lemma} 
\label{lem:caseC-easy-right-half-comm-rel}
\begin{thmenumerate}
\item
$\eta_1, \eta_2, \ldots, \eta_m$ are anti-commutative
{\upshape(}i.e. $\eta_s \eta_t + \eta_t \eta_s = 0${\upshape)}.
\item
$\zeta_1, \zeta_2, \ldots, \zeta_m$ are anti-commutative.
\item
$\eta_s \zeta_t + \zeta_t \eta_s = \delta_{s,t} \sigma$
for $1 \le s,t \le m$.
\item
$[\Lambda, \eta_t] = \eta_t \sigma$ for $1 \le t \le m$.
\end{thmenumerate}
\end{lemma}
\begin{proof}
(1)
The entries of $P = (X, \partial^Y)$ commute with each other,
and $f_i$ and $f_j$ are anti-commutative.
Hence $\eta_s$ and $\eta_t$ are anti-commutative.
The statement (2) can be proved similarly.

(3)
By Equation~(\ref{eq:comm-rel-of-P-Q}), we get 
$\eta_s \zeta_t + \zeta_t \eta_s 
=
\sum_{i,j=1}^{p+q} f_i f^\ast_j [P_{s,i}, Q_{t,j}]
=
\sum_{i,j} f_i f^\ast_j \delta_{s,t} \veps_{i,j}
=
\delta_{s,t} \sigma$.

(4) 
Note that 
$\Lambda \eta_t = 
\sum_{s=1}^m \eta_s \zeta_s \eta_t =
\sum_s \eta_s (-\eta_t \zeta_s + \delta_{s,t}\sigma) =
\eta_t \Lambda + \eta_t \sigma$.
Hence,  $[\Lambda, \eta_t] = \eta_t \sigma$.
\end{proof}
\begin{lemma} 
\label{lem:Lambda-to-Detpq}
Put $B = \T P Q$ and denote $\Lambda(u) = \Lambda - u \sigma$.  
Then $\Lambda = \sum_{i,j=1}^{p+q} f_i f^\ast_j B_{i,j}$.
We have the following relation between $\Lambda$ and $\Det_{p,q}$.
\begin{gather*}
  \Lambda(u_1) \Lambda(u_2) \cdots \Lambda(u_d) =
\textstyle
  d! (-1)^{\frac{d(d-1)}{2}} \!\!\!\! \sum\limits_{I, J \in \comb{p+q}{d}} \!\! f_I f^\ast_J
  \Det_{p,q}( B_{I,J}; u_1, u_2, \ldots, u_d).
\end{gather*}
\end{lemma}
\begin{proof}
We can compute as follows.
  \begin{align*}
    \lefteqn{ \Lambda(u_1) \cdot \Lambda(u_2) \cdots \Lambda(u_d) }
    \\ & =
\textstyle \!\!\!
    \sum\limits_{i_1,\ldots,i_d=1}^{p+q}
    \sum\limits_{j_1,\ldots,j_d=1}^{p+q} 
    \!\!
    f_{i_1}f^\ast_{j_1} \cdots f_{i_d}f^\ast_{j_d} 
    (B_{i_1,j_1}-u_1\veps_{i_1,j_1}) \cdots
    (B_{i_d,j_d}-u_d\veps_{i_d,j_d}) 
    \\ & =
\textstyle
    \sum_{I,J\in\comb{p+q}{d}} \sum_{\sigma,\tau\in\Symm_d}
    \sgn(\sigma) \sgn(\tau) (-1)^{d(d-1)/2} f_I f^\ast_J
    \\ & \hspace{10ex}
    \times
    (B_{i_{\sigma(1)},j_{\tau(1)}} -
      u_1\veps_{i_{\tau(1)},j_{\sigma(1)}}) \cdots
    (B_{i_{\sigma(d)},j_{\tau(d)}} -
      u_d\veps_{i_{\tau(d)},j_{\sigma(d)}}) 
    \\ &=
\textstyle
    d! (-1)^{d(d-1)/2} \sum_{I,J\in\comb{p+q}{d}} 
    f_I f^\ast_J \Det_{p,q}(B_{I,J}; u_1, \ldots, u_d).
  \end{align*}
\end{proof}

\begin{proof}[Proof of Proposition~\ref{prop:caseC-easy-right-half}]
For $I, J \in \comb{p+q}{d}$, we show that
\begin{gather}
  \label{eq:proof:caseC-easy-right-half}
\textstyle
  \sum_{S\in\comb{m}{d}} \det P_{S,I} \det Q_{S,J}
  =
  \Det_{p,q}( (\T P Q)_{I,J}; d-1, d-2, \ldots, 0).
\end{gather}
Then Proposition~\ref{prop:caseC-easy-right-half} follows 
from Equation~(\ref{eq:caseC-image-of-B}).
To prove it, we compute 
$\sum_{S \in \comb{m}{d}} \eta_S \zeta_S$
in two different ways.
First, we compute as follows.
\begin{align*}
\textstyle
\sum_{S \in \comb{m}{d}} \eta_S \zeta_S
=
\sum_{S \in \comb{m}{d}} 
\sum_{I,J \in \comb{p+q}{d}} f_I f^\ast_J
\det P_{S,I} \det Q_{S,J}.
\end{align*}
The coefficient of $f_I f^\ast_J$ is equal to the left-hand side
of Equation~(\ref{eq:proof:caseC-easy-right-half}).

Second we compute using 
Lemma~\ref{lem:caseC-easy-right-half-comm-rel}~(4)
as follows:
\begin{align*}
\textstyle
  \sum_{S\in\comb{m}{d}} \eta_S \zeta_S
  &=
\textstyle
  \sum\limits_{s_1, \ldots, s_{d}=1}^m \dfrac{1}{d!} 
  \eta_{s_1} \cdots \eta_{s_{d}} \cdot
  \zeta_{s_1} \cdots \zeta_{s_{d}}
  \\ &=
  \frac{(-1)^{d-1}}{d!} 
\textstyle
  \sum\limits_{s_1, \ldots, s_{d-1}=1}^m 
  \eta_{s_1} \cdots \eta_{s_{d-1}} \cdot
  \Lambda \cdot
  \zeta_{s_1} \cdots \zeta_{s_{d-1}}
  \\ &=
  \frac{(-1)^{d-1}}{d!} 
\textstyle
  \sum\limits_{s_1, \ldots, s_{d-1}=1}^m 
  \Lambda(d-1) \cdot
  \eta_{s_1} \cdots \eta_{s_{d-1}} \cdot
  \zeta_{s_1} \cdots \zeta_{s_{d-1}}.
\end{align*}
In this way we repeat producing $\Lambda$ from $\eta_s \zeta_s$
and moving it to the left,
and we have
\begin{align*}
\textstyle
  \sum_{S\in\comb{m}{d}} \eta_S \zeta_S
  &=
  \frac{(-1)^{d(d-1)/2}}{d!}
  \Lambda(d-1) \Lambda(d-2) \cdots \Lambda(0)
  \\ &=
\textstyle
  \sum_{I,J\in\comb{p+q}{d}} f_I f^\ast_J
  \Det_{p,q}((\T P Q)_{I,J}; d-1, d-2, \ldots, 0),
\end{align*}
by using Lemma~\ref{lem:Lambda-to-Detpq}.
The coefficient of $f_I f^\ast_J$ is equal to the right-hand side
of Equation~(\ref{eq:proof:caseC-easy-right-half}).
Thus we have proved Equation~(\ref{eq:proof:caseC-easy-right-half}),
and hence Proposition~\ref{prop:caseC-easy-right-half}.
\end{proof}

\subsection{Proof of the Capelli identity for Case $\CC$ {\upshape (2)}}
\label{subsec:C-2-pf}

Here we shall prove Proposition~\ref{prop:caseC-not-easy-left-half}.
We have already given the proof of
Proposition~\ref{prop:caseC-not-easy-right-half},
and Theorem~\ref{thm:caseC-not-easy}
is a direct consequence of these propositions.
We use the same setting as in \S~\ref{subsec:C-1-pf}.
Let $e_s$ and $e^\ast_s$ be the elements in the standard basis of
$\CC^m$ and $(\CC^m)^\ast$ respectively.
We again define several elements in 
$\bigwedge (\CC^m \oplus (\CC^m)^\ast) \otimes_\CC \PD(V)$.
The elements $\veps = - I_{p,q}, \alpha_j, \beta_j$ and $\tau$ 
are the same as before,
and newly defined elements are denoted with tilde.
\begin{gather*} 
\textstyle
  \alpha_j = \sum_{s=1}^m e_s P_{s,j},
  \;
  \beta_j = \sum_{s=1}^m e^\ast_s Q_{s,j}
  \; (1 \le j \le p+q) ; 
  \quad
  \tau = \sum_{s=1}^m e_s e^\ast_s,
  \\
  \talpha_j = \veps_{j,j} \alpha_j,
  \quad
  \tbeta_j = \veps_{j,j} \beta_j
  \quad (1 \le j \le p+q),
  \\
  \Xi
\textstyle
  = \!\! \sum\limits_{s,t=1}^m e_s e^\ast_t (P \T Q)_{(s,t)}
  = \!\! \sum\limits_{j=1}^{p+q} \!\! \alpha_j \beta_j
  = \!\! \sum\limits_{j=1}^{p+q} \!\! \talpha_j \tbeta_j,
\quad
  \tXi = \!\! \sum\limits_{j=1}^{p+q} \!\! \talpha_j \beta_j 
  = \!\! \sum\limits_{j=1}^{p+q} \!\! \alpha_j \tbeta_j.
\end{gather*}

\begin{lemma}
We have the following relations for the elements above.
\begin{thmenumerate}
\item
$2(p+q)$ elements
$ \{ \alpha_i , \talpha_i ; 1 \le i \le p+q \} $ are anti-commutative.
\item
Similarly, 
$ \{ \beta_i , \tbeta_i ; 1 \le i \le p+q \} $ are anti-commutative.
\item
We have the following commutation relations:
\begin{align}
  \label{eq:comm-rel-of-Xi-tXi}
  & [\Xi, \tXi] = 0,
  \\
  \label{eq:comm-rel-of-Xi-beta}
  & [ \Xi, \alpha_i ] = \talpha_i \tau,
  \quad
  [ \Xi, \beta_i ] = -\tbeta_i \tau
  \qquad (1 \le i \le p+q).
\end{align}
\end{thmenumerate}
\end{lemma}
\begin{proof}
The statements (1) and (2) follow from
Equation~(\ref{eq:comm-rel-of-alpha-alpha}) and
the definition of $\talpha_j$ and $\tbeta_j$.

(3)-Equation~(\ref{eq:comm-rel-of-Xi-tXi}):
Since $\Xi = \Xi_X + \Xi_Y$ and $\tXi = -\Xi_X + \Xi_Y$,
the assertion follows from Equation~(\ref{eq:comm-rel-of-XiX-XiY}).

(3)-Equation~(\ref{eq:comm-rel-of-Xi-beta}):
Again, since $\Xi = \Xi_X + \Xi_Y$ and $\tXi = -\Xi_X + \Xi_Y$,
the first assertion follows from 
Equations~(\ref{eq:comm-rel-of-alpha-XiX}) and
(\ref{eq:comm-rel-of-alpha-XiY}).
The second assertion can be proved similarly.
\end{proof}

For integers $ u $ and $ v $, 
let us define one more element $\gamma(u,v)$ in
$\bigwedge (\CC^m \oplus (\CC^m)^\ast) \otimes_\CC \PD(V)$.
\begin{gather*}
  \gamma(u,v) =
  \sum_{1 \le j_1,\ldots,j_{u+v} \le p+q}
  \alpha_{j_1} \cdots \alpha_{j_u} \cdot
  \talpha_{j_{u+1}} \cdots \talpha_{j_{u+v}} \cdot
  \beta_{j_{u+v}} \cdots \beta_{j_1},
  \\
  \gamma(0,0) = 1, 
  \qquad 
  \gamma(u,v) = 0 \quad (\text{$u<0$ or $v<0$}).
\end{gather*}
If we change the order of $\alpha_j$ or $\talpha_j$
\emph{and} the corresponding order of $\beta_j$ at the same time, 
then the signs will be canceled.  
Thus in the definition above,
we can put tildes to any $v$ positions out of 
$\alpha_{j_1}, \ldots, \alpha_{j_{u+v}}$.
Moreover we can move a tilde from $\alpha_j$ to $\beta_j$,
since $\talpha_j = \veps_{j j} \alpha_j$ 
and $\tbeta_j = \veps_{j j} \beta_j$.
We also note that $\gamma(1,0) = \Xi$ and
$\gamma(0,1) = \tXi$.

\begin{lemma}
\label{lem:gamma-recurrence}
We have the following relations of $\gamma(u,v)$ with $\Xi$.
\begin{thmenumerate}
\item
For non-negative integers $u$ and $v$, we have
\begin{gather*}
 \gamma(u,v) \Xi =
\gamma(u+1, v) + u\tau\gamma(u-1, v+1) + v\tau\gamma(u+1, v-1).
\end{gather*}
\item
For $d \ge 0$,
$\Xi^d$ can be expressed as an integral linear combination of
$\gamma(u, v) \tau^{d-u-v}$, 
where $u, v$ are non-negative integers and $d-u-v \ge 0$.
\item
$\Xi$, $\tXi$, $\tau$ and $\gamma(u, v)$ commute with each other.
\item
$\gamma(u,v)$ and $\gamma(u',v')$ commute with each other.
\end{thmenumerate}
\end{lemma}
\begin{proof}
(1)
We have
\begin{align}
  \nonumber
  &
  \gamma(u,v) \Xi =
\textstyle
  \sum\limits_{1 \le j_1,\ldots,j_{u+v} \le p+q}
  \alpha_{j_1} \cdots \alpha_{j_u} \cdot
  \talpha_{j_{u+1}} \!\! \cdots \talpha_{j_{u+v}} \!\! \cdot
  \beta_{j_{u+v}} \!\! \cdots \beta_{j_1}
  \cdot
  \Xi
  \\ &=
  \nonumber
\textstyle
  \sum\limits_{j_1,\ldots,j_{u+v}}
  \alpha_{j_1} \cdots \alpha_{j_u} \cdot
  \talpha_{j_{u+1}} \cdots \talpha_{j_{u+v}} \!\! \cdot
  \Xi \cdot
  \beta_{j_{u+v}} \!\! \cdots \beta_{j_1}
  \\ &\qquad 
  \label{eq:proof:lem-gamma-rec}
  + 
\textstyle
\!\! \sum\limits_{j_1,\ldots,j_{u+v}} \sum\limits_{l=1}^{u+v} \!\! 
  \alpha_{j_1} \cdots \alpha_{j_u} \!\! \cdot
  \talpha_{j_{u+1}} \cdots \talpha_{j_{u+v}} \!\! \cdot
  \beta_{j_{u+v}} \cdots [\beta_{j_l}, \Xi] \cdots \beta_{j_1}.
\end{align}
Since $\Xi = \sum_i \alpha_i \beta_i$,
the first term of Expression~(\ref{eq:proof:lem-gamma-rec})
is equal to $\gamma(u+1, v)$.
For each summand of the second term in 
Expression~(\ref{eq:proof:lem-gamma-rec}),
there appears $[\beta_{j_l}, \Xi] = \tbeta_{j_l} \tau$
(cf.~Equation~(\ref{eq:comm-rel-of-Xi-beta})),
and this tilde over $\beta_{j_l}$ can be moved to $\alpha_{j_l}$
($1 \le l \le u$),
or cancel with $\talpha_{j_l}$ ($u+1 \le l \le u+v$).
Therefore summands of the second term 
of (\ref{eq:proof:lem-gamma-rec}) is equal to
$\gamma(u-1,v+1)\tau$ for $1 \le l \le u$,
and equal to $\gamma(u+1,v-1)\tau$ for $u+1 \le l \le u+v$.
Thus, after summarizing the first term and the second term,
we have
\begin{gather*}
  \gamma(u,v) \Xi = 
  \gamma(u+1,v) + u\tau\gamma(u-1,v+1) + v\tau\gamma(u+1,v-1).
\end{gather*}

(2)
Starting with $\Xi = \gamma(1,0)$,
we repeat multiplying $\Xi$ from the right to this equation,
and use (1) of this lemma.  
Then it is shown inductively that $\Xi^d$ can be expressed
as a linear combination of $\gamma(u, v) \tau^{d-u-v}$
with integral coefficients.

(3)
Note that $\tau$ is central, and that $\Xi$ commutes with $\tXi$
by Equation~(\ref{eq:comm-rel-of-Xi-tXi}).
We prove that $\gamma(u,v)$ can be expressed by
$\Xi$, $\tXi$ and $\tau$, which will complete the proof.  
It follows from (1) of this lemma that
\begin{gather*}
  \gamma(u+1, v) =
  \gamma(u,v) \Xi  - u\tau\gamma(u-1, v+1) - v\tau\gamma(u+1, v-1).
\end{gather*}
Therefore $\gamma(u,v)$ can be written by using 
$\Xi$, $\tau$ and $\gamma(u',v')$ ($u'+v'<u+v$).
Repeat this operation, then 
it turns out that $\gamma(u,v)$ has an expression of
$\Xi = \gamma(1,0)$, $\tXi = \gamma(0,1)$ and $\tau$.

(4) 
Since $\gamma(u,v)$ is expressed
by $\Xi$, $\tXi$ and $\tau$ by the proof of (3) above,
$\gamma(u,v)$ and $\gamma(u',v')$ are commutative.
\end{proof}

\begin{definition} 
\label{defn:b-d-uv}
Thanks to Lemma~\ref{lem:gamma-recurrence}~(2),
we can define integers $b^d_{u,v}$ by
\begin{gather*}
\textstyle
  \Xi^d =
  \sum_{0 \le u+v \le d} b^d_{u,v} \gamma(u,v) \tau^{d-u-v}
  \qquad
(d \ge 0).
\end{gather*}
For non-negative integers $p,q,u,v$ and $J \in \comb{p+q}{u+v}$,
we set
\begin{align*}
  \veps(J; u, v) =
  \veps(\alpha, \beta; u, v) &= 
\textstyle
  \sum\limits_{\sigma\in\Symm_{u+v}}
  \veps_{j_{\sigma(u+1)},j_{\sigma(u+1)}} \cdots
  \veps_{j_{\sigma(u+v)},j_{\sigma(u+v)}},
  \\
  \veps(\emptyset; 0,0) &= 
  \veps(0,0; 0,0) = 1,
\end{align*}
where $\alpha$ and $\beta$ are defined in 
Equation~\eqref{eq:def-of-c^d_J-alpha-beta}.  
When $u$ or $v$ is negative, 
we put $b^d_{u,v} = \veps(J; u,v) = 0$.
\end{definition}

\begin{lemma}
\label{lem:gamma-to-det}
For non-negative integers $p,q,u,v$ and $ w = u + v $, we have 
\begin{align*}
  \gamma(u,v) = (-1)^{\frac{w(w-1)}{2}}
\textstyle
  \sum\limits_{J \in \comb{p+q}{w}} \veps(J;u,v)
  \sum\limits_{S,T \in \comb{m}{w}} e_S e^\ast_T
  \det P_{S,J} \det Q_{T,J}.
\end{align*}
\end{lemma}
\begin{proof}
By the definition of $\gamma(u,v)$, we have
\begin{align*}
&
  \gamma(u,v) =
\textstyle
  \!\! \sum\limits_{1\le j_1, \ldots, j_{w} \le p+q} \!\! 
  \alpha_{j_1} \cdots \alpha_{j_u} \cdot
  \talpha_{j_{u+1}} \cdots \talpha_{j_{w}} \cdot
  \beta_{j_{w}} \cdots \beta_{j_1}
  \\ &=
\textstyle
  \!\! \sum\limits_{1\le j_1, \ldots, j_{w} \le p+q} \!\! 
  \veps_{j_{u+1},j_{u+1}} \cdots \veps_{j_{w},j_{w}} \cdot
  \alpha_{j_1} \cdots \alpha_{j_{w}} \cdot
  \beta_{j_{w}} \cdots \beta_{j_1}
  \\ &=
\textstyle
  \!\! \sum\limits_{\substack{J\in\comb{p+q}{w}\\ \sigma\in\Symm_{w}}} \!\! 
  \veps_{j_{\sigma(u+1)},j_{\sigma(u+1)}} \cdots 
  \veps_{j_{\sigma(w)},j_{\sigma(w)}} \!\! \cdot 
  \sgn(\sigma) \alpha_{j_1} \cdots \alpha_{j_{w}} \!\! \cdot 
  \sgn(\sigma) \beta_{j_{w}} \cdots \beta_{j_1}, 
\end{align*}
where $ w = u + v $.  
In this expression, the sum over $\sigma\in\Symm_{w}$ 
of the products of $\veps_{j,j}$
is equal to $\veps(J;u,v)$,
and there appears $(-1)^{w(w-1)/2}$
by sorting $\beta_{j_a}$'s into increasing order in $a$.
Thus it follows from Equation~(\ref{eq:alpha-beta-causes-det})
that the expression gives our desired formula.
\end{proof}

\begin{proof}[Proof of Proposition~\ref{prop:caseC-not-easy-left-half}]
%
We prove Proposition~\ref{prop:caseC-not-easy-left-half}
by computing $\Xi^d$ in two different ways,
and make a contraction using
$e_S e^\ast_T \mapsto \delta_{S,T}$,
where $\delta_{S,T}$ denotes Kronecker's delta.
First, we get the following formula, 
\begin{gather}
  \label{eq:proof:Xi-d-computation-1}
  \Xi^d =
\textstyle
  \sum_{S,T\in\comb{m}{d}} (-1)^{\frac{d(d-1)}{2}} e_S e^\ast_T \, 
  d! \Det( ( P \T Q)_{S,T} ),
\end{gather}
by a similar computation to Lemma~\ref{lem:caseC-Xi-to-Det}.
We apply the contraction $e_S e^\ast_T \mapsto \delta_{S,T}$
to the right-hand side of 
Equation~(\ref{eq:proof:Xi-d-computation-1}),
and we have
\begin{gather*}
\textstyle
  (-1)^{d(d-1)/2} d! \sum_{S\in\comb{m}{d}} 
  \Det((P \T Q)_{S,S} ),
\end{gather*}
which is equal to 
$ \omega(X_d) $ 
multiplied by $(-1)^{d(d-1)/2} d!$ 
in view of Equation~(\ref{eq:caseC-image-of-EX+tEY}).

Second, we compute in the following way.  
It follows from Definition~\ref{defn:b-d-uv}
and Lemma~\ref{lem:gamma-to-det} that
\begin{align*}
  \Xi^d &=
\textstyle
  \sum\limits_{l=0}^d \sum\limits_{u+v=l}
  b^{d}_{u,v} 
  (-1)^{\frac{l(l-1)}{2}}
  \sum_{J \in \comb{p+q}{l}} \veps(J;u,v) \cdot \Delta_J \cdot
  \tau^{d-l},
\\
\intertext{where we temporally put }
\Delta_J &=  
\textstyle
\sum_{S,T \in \comb{m}{l}} e_S e^\ast_T
  \det P_{S,J} \det Q_{T,J} .
\end{align*}
Since 
$\tau^k = (\sum_{s=1}^m e_s e^\ast_s)^k = 
k! \sum_{U\in\comb{m}{k}} (-1)^{k(k-1)/2} e_U e^\ast_U$,
the expression above is equal to
\begin{align*}
&
\textstyle
    \sum\limits_{l=0}^d \sum\limits_{u+v=l}
    \sum\limits_{\substack{J \in \comb{p+q}{l} \\
        S,T \in \comb{m}{l}, U\in\comb{m}{d-l}}}
    b^{d}_{u,v} 
    (-1)^{\frac{l(l-1)}{2}} \veps(J;u,v) e_S e^\ast_T
    \det P_{S,J} \det Q_{T,J} 
  \\[-3ex]
 & \hspace*{.5\textwidth}
 \times (d-l)! (-1)^{\frac{1}{2}(d-l)(d-l-1)} e_U e^\ast_U
  \\ & = (-1)^{\frac{d(d-1)}{2}} \!\!\!
\textstyle
  \sum\limits_{l, u, v, J, S, T, U}
  b^{d}_{u,v} 
  \veps(J;u,v) 
  (d-l)! \; 
  e_S e_U \, e^\ast_T e^\ast_U \,
  \det P_{S,J} \det Q_{T,J}. 
\end{align*}
When applying the contraction $e_S e^\ast_T \mapsto \delta_{S,T}$
to this expression,
only the summands with $S=T$ do not vanish.
If we fix $S (=T)$,
then the number of choices of $U$ is $\binom{m-l}{d-l}$,
since $U$ should not intersect $S$.
Thus we contract the expression above, and obtain
\begin{gather*}
(-1)^{\frac{d(d-1)}{2}} \!
\textstyle
  \sum\limits_{l=0}^d \sum\limits_{u+v=l}
  \sum\limits_{\substack{J \in \comb{p+q}{l} \\ S \in \comb{m}{l}}}
  b^d_{u,v} \veps(J;u,v) (d-l)! \displaystyle \binom{m-l}{d-l}
  \det P_{S,J} \det Q_{S,J} 
  \\ =
\textstyle
  \sum\limits_{l=0}^d
  \left( \sum\limits_{u+v=l} b^d_{u,v} \veps(J;u,v) \right)
  \sum\limits_{\substack{J \in \comb{p+q}{l} \\ S \in \comb{m}{l}}}
  \dfrac{(-1)^{\frac{d(d-1)}{2}} (m-l)!}{(m-d)!}
  \det P_{S,J} \det Q_{S,J},
\end{gather*}
which is equal the the right-hand side of 
Proposition~\ref{prop:caseC-not-easy-left-half}
multiplied by $(-1)^{d(d-1)/2}d!$
except that $c^d_J$ is replaced with
$\sum_{u+v=l} b^d_{u,v} \veps(J;u,v)$.

We prove $c^d_J = \sum_{u+v=l} b^d_{u,v} \veps(J;u,v)$
in Appendix~\ref{sec:c^d_J},
and this completes the proof of 
Proposition~\ref{prop:caseC-not-easy-left-half}.
\end{proof}

\section{Case $\RR$}
\label{sec:R}

Here we give the Capelli identities for symmetric pairs
for Case $\RR$ in Table  (\ref{table:see-saw-pairs}) without proof.
There are two different types of the Capelli identities
as remarked at the beginning of Section~\ref{sec:C}.
In Theorem~\ref{thm:caseR-easy}
we give the first form,
which is simpler as differential operators; 
in Theorem~\ref{thm:caseR-not-easy}
we give the second form,
which has a simpler Capelli element $X_d \in U(\frakg)^K$.
In contrast to Case $\CC$,
we only have the identities expressing 
$\omega(X_d) \in \PD(V)^{K\times H}$
in explicit differential operators
for Capelli elements $X_d \in U(\frakg)^K$.
In other words, we only have identities corresponding to 
the left half of the picture below.
\begin{align*}
\begin{array}{ccccccc}
U(\frakg)^K & \overset{\omega}{\longrightarrow} 
& \PD(V)^{K\times H} &
\overset{\omega}{\longleftarrow} & U(\frakm)^H 
\\
|| &&&& ||\\
\makebox[0pt][c]{
  $U(\frakgl_m(\CC))^{O_m(\CC)}$}
&&&&
\makebox[0pt][c]{$U(\fraksp_{2n}(\CC))^{GL_n(\CC)}$}
\end{array}
\end{align*}

\subsection{Formulas for the Weil representation}

Recall our see-saw pair:
\begin{align*}
\begin{array}{cccccccc}
\frakg_0 &=& \fraku_m      & & \fraksp_{2n}(\RR) &=& \frakm_0 \\
&&           \cup          &   \text{\LARGE$\times$} & \cup \\
\frakk_0 &=& \frako_m(\RR) & & \fraku_n          &=& \frakh_0
\end{array}
\end{align*}
Here $\frakg_0 = \fraku_m$ is realized as 
the set of the $m\times m$ skew Hermitian matrices,
and $\frakk_0 = \frako_m(\RR)$ is the set of the
alternating matrices.
Let $\frakp_0 = \sqrt{-1}\Sym_m(\RR) \subset \fraku_m$,
and we have the direct sum decomposition 
$\frakg_0 = \frakk_0 \oplus \frakp_0$.
The Lie algebra
$\frakm_0 = \fraksp_{2n}(\RR)$ is defined as
\begin{align*}
\fraksp_{2n}(\RR) = 
\left\{
\begin{pmatrix} A & B \\ C & -\T A \end{pmatrix} 
\;;\;
\begin{array}{l}
  A \in \frakgl_n(\RR), \\ B, C \in \Sym_n(\RR)
\end{array}
\right\},
\end{align*}
and $\frakh_0 = \fraku_n$ is the Lie algebra of a maximal compact subgroup, 
which is embedded into $ \fraksp_{2mn}(\RR) $ just as $ \frakg_0 $ below.
%
\begin{align*}
  \begin{array}{cccccccc}
    \frakg_0 &=& \fraku_m & \hookrightarrow & \fraksp_{2mn}(\RR) \\
    && A + \sqrt{-1}B &\mapsto &
    \begin{pmatrix}
      A \ast 1_n & -B \ast 1_n \\
      B \ast 1_n & A \ast 1_n
    \end{pmatrix}
    & \quad
    (A, B: \text{~real matrices}), \\[10pt]
    \frakm_0 &=& \fraksp_{2n} & \hookrightarrow & \fraksp_{2mn}(\RR) \\
    && \begin{pmatrix} A & B \\ C & -\T A \end{pmatrix} 
    & \mapsto&
    \begin{pmatrix}
      A^{\oplus m} & B^{\oplus m} \\
      C^{\oplus m} & -\T A^{\oplus m} 
    \end{pmatrix}.
  \end{array}
\end{align*}

Let $V = \Mat(m, n; \CC)$, and 
$x_{s,i}$ its linear coordinate system,
and put $\partial_{s,i} = \partial/\partial x_{s,i}$.
Denote the Weil representation of $\fraksp_{2mn}(\CC)$ 
on the polynomial ring $\CC[V]$ by $\omega$.
Through the embeddings into $\fraksp_{2mn}(\CC)$,
the complexified Lie algebras
$\frakg = (\frakg_0) \otimes_\RR \CC$ and
$\frakm = (\frakm_0) \otimes_\RR \CC$
act on $\CC[V]$.
We denote these representations also by $\omega$.
Let $E_{s,t} \in \frakg = \frakgl_m(\CC)$ be the matrix unit.
Then its representation through $\omega$ is given by
(see (4.5) of \bracketcite{MR1845714}, e.g.)
\begin{align*}
\textstyle
\omega(E_{s,t}) = 
\sum_{i=1}^n x_{s,i} \partial_{t,i} + \dfrac{n}{2} \delta_{s,t}.
\end{align*}
Note that $x_{s,i}$ of this article
corresponds to $x_{n(s-1)+i}$ in \bracketcite{MR1845714}.
%
%
To give the representation of $\frakm = \fraksp_{2n}(\CC)$,
we define the following elements in $\frakm$.
\begin{align*}
X_{\veps_i-\veps_j} &= 
\dfrac{1}{2} \begin{pmatrix}
F_{i,j} & \sqrt{-1} G_{i,j} \\
-\sqrt{-1} G_{i,j} & F_{i,j}
\end{pmatrix}
& (1 \le i,j \le n), \\
X_{\pm (\veps_i+\veps_j)} &= 
\dfrac{1}{2} \begin{pmatrix}
G_{i,j} & \mp \sqrt{-1} G_{i,j} \\
\mp \sqrt{-1} G_{i,j} & -G_{i,j}
\end{pmatrix}
& (1 \le i,j \le n), 
\end{align*}
where $F_{i,j} = E_{i,j} - E_{j,i}$
and $G_{i,j} = E_{i,j} + E_{j,i}$ are elements in $\Mat(n; \CC)$.
Note that we do not write $X_0$, but $X_{\veps_i-\veps_i}$
to distinguish $X_{\veps_i-\veps_i}$ from $X_{\veps_j-\veps_j}$.
Then $X_{\veps_i-\veps_i}$ ($1 \le i \le n$) 
form a basis of a Cartan subalgebra of $\frakm = \fraksp_{2n}(\CC)$,
and the other elements defined above are the root vectors 
with respect to this Cartan subalgebra.
The actions of these elements through $\omega$ are given as follows:
\begin{align*}
\omega(X_{\veps_i-\veps_j}) &=
\textstyle
\sum\limits_{s=1}^m x_{s,i} \partial_{s,j} + \frac{m}{2} \delta_{i,j}
& (1 \le i,j \le n), \\
\omega(X_{\veps_i+\veps_j}) &=
\textstyle
\sum\limits_{s=1}^m x_{s,i} x_{s,j} ,
\quad
\omega(X_{-\veps_i-\veps_j}) =
\sum_{s=1}^m \partial_{s,i} \partial_{s,j}
& (1 \le i,j \le n).
\end{align*}
Note that we used the normalization 
$x_{s,i} \mapsto \sqrt{2} x_{s,i}$, which is slightly different from 
(4.5) of \bracketcite{MR1845714}.
We define matrices of differential operators by
\begin{align*}
X &=  (x_{s,i})_{\substack{1 \le s \le m \\ 1 \le i \le n}} , 
\quad &
\partial &= (\partial/\partial 
x_{s,i})_{\substack{1 \le s \le m \\ 1 \le i \le n}}
&& \in \Mat(m,n; \PD(V)),
\\
P &=  (X, \partial) ,
\quad &
Q &=  (\partial, X) 
&& \in \Mat(m,2n; \PD(V)),
\\
\EE &= (E_{s,t})_{1 \le s,t \le m}
&&&& \in \Mat(m; U(\frakg)),
\end{align*}
and we use the notation like $\omega(\EE)$
of a matrix form as used in Case $\CC$.
For example, we write 
\begin{math}
\omega(\EE + \T\EE) = P \T Q.
\end{math}
It is known that
\begin{align*}
\textstyle
  \sum_{S\in\comb{m}{d}} \det(\EE + \T\EE)_{S,S}
  \in S(\frakp)^K
  \quad (d = 1, 2, \ldots, m)
\end{align*}
is a generating set of $S(\frakp)^K$,
where $\EE$ is considered as a matrix with entries
$E_{s,t}$ in $S(\frakg)$.

\subsection{Capelli identity for Case $\RR$}

We have two different types of the Capelli identities for Case $\RR$.
\begin{theorem}
\label{thm:caseR-easy}
For $d \ge 1$,
define
$X_d \in U(\frakg)^K =
U(\frakgl_m(\CC))^{O_m(\CC)}$
by
\begin{align*}
  & X_d =
\textstyle
  \sum\limits_{S\in\comb{m}{d}}
  \sum\limits_{l=0}^d
  \sum\limits_{\substack{
      S', T' \\
      S\dblprime, T\dblprime \\
  }}
  (\pm 1)
  \det(\EE_{S',T'};
    l-1-\frac{n}{2}, l-2-\frac{n}{2}, \ldots, -\frac{n}{2})
  \\
  & \hspace*{.15\textwidth}
\textstyle
  \times \det((\T\EE)_{S\dblprime,T\dblprime};
    -(d-l-1)+\frac{n}{2}, -(d-l-2)+\frac{n}{2}, \ldots, \frac{n}{2}),
\end{align*}
where the third summation is taken over 
$      S', T' \in \comb{m}{l} $ and 
$      S\dblprime, T\dblprime \in \comb{m}{d-l} $ 
such that 
$ S = S' \amalg S\dblprime = T' \amalg T\dblprime $, 
and the signature is given by $ (\pm1) = (-1)^{\scriptscriptstyle l(S',S\dblprime)+l(T',T\dblprime)} $.
%
Then we have the Capelli identity
\begin{align*}
\textstyle
  \omega(X_d) = 
  \sum_{S\in\comb{m}{d}}
  \sum_{J\in\comb{2n}{d}}
  \det P_{S,J} \det Q_{S,J}.
\end{align*}
%
\end{theorem}
\begin{theorem}
\label{thm:caseR-not-easy}
For $d \ge 1$,
define
$X_d \in U(\frakg)^K =
U(\frakgl_m(\CC))^{O_m(\CC)}$
by
\begin{align*}
\textstyle
  X_d = \sum_{S\in\comb{m}{d}} \Det(\EE+\T\EE)_{S,S},
\end{align*}
where $ \Det $ denotes the symmetrized determinant.
%
Then we have the Capelli identity
\begin{align*}
  \omega( X_d )
  &= 
\textstyle
  \sum\limits_{l=0}^d \dfrac{(m-l)!}{d!(m-d)!}
  \sum_{J\in\comb{2n}{l}} c^d_J
  \sum_{S\in\comb{m}{l}} \det P_{S,J} \det Q_{S,J},
\end{align*}
where $c^d_J$ is defined in Equation~\eqref{eq:def-of-c^d_J} 
with both $p$ and $q$ replaced by $n$.
%
\end{theorem}

\section{Case $\HH$}
\label{sec:H}

In this section, we give the Capelli identities for symmetric pairs
for Case $\HH$ in Table  (\ref{table:see-saw-pairs}) without proof.
%
As in the case of $ \RR $, 
for Capelli elements $X_d \in U(\frakg)^K$, 
we only have the identities expressing 
$\omega(X_d) \in \PD(V)^{K\times H}$
by explicit differential operators.  
In other words, we only consider 
the left half of the picture below.
\begin{align*}
\begin{array}{ccccccc}
U(\frakg)^K & \overset{\omega}{\longrightarrow} 
& \PD(V)^{K\times H} &
\overset{\omega}{\longleftarrow} & U(\frakm)^H 
\\
|| &&&& ||\\
\makebox[0pt][c]{
  $U(\fraku_{2m})^{USp_m}$}
&&&&
\makebox[0pt][c]{$U(\frako^*_{2n})^{U_n}$}
\end{array}
\end{align*}

\subsection{Formulas for the Weil representation}

Our see-saw pair in this section is:
\begin{align*}
\begin{array}{cccccccc}
\frakg_0 &=& \fraku_{2m} & & \frako^\ast_{2n} &=& \frakm_0 \\
&&           \cup     & \text{\LARGE$\times$} & \cup \\
\frakk_0 &=& \frakusp_m & &        \fraku_n &=& \frakh_0
\end{array}
\end{align*}
Here $\frakg_0, \frakk_0, \frakm_0$ and $\frakh_0$ are
realized as follows.

\begin{align*}
  \frakg_0 = \fraku_{2m} &=
  \{ 2m \times 2m \text{~skew Hermitian matrices} \},
  \\
  \frakk_0 = \frakusp_m &=
  \left\{
    \begin{pmatrix} X & -Y \\ \BAR{Y} & \BAR{X} \end{pmatrix}
    \;;\;
    \begin{array}{l}
      X: \text{$m\times m$ skew Hermitian}, \\
      Y \in \Sym_m(\CC)
    \end{array}
  \right\}
  \subset \frakg_0,
  \\
  \frakm_0 = \frako^\ast_n &=
  \left\{
    \begin{pmatrix} X & -Y \\ \BAR{Y} & \BAR{X} \end{pmatrix}
    \;;\;
    \begin{array}{l}
      X: \text{$n\times n$ skew Hermitian}, \\
      Y \in \Alt_n(\CC)
    \end{array}
  \right\},
  \\
  \frakh_0 = \fraku_n &=
  \left\{
    \begin{pmatrix} X & 0 \\ 0 & \BAR{X} \end{pmatrix}
    \;;\;
    X: \text{$n\times n$ skew Hermitian}
  \right\} \subset \frakm_0 .
\end{align*}
Define $\frakp_0$ by 
\begin{gather*}
  \frakp_0 =
  \left\{ 
    \begin{pmatrix} X & Y \\ \BAR{Y} & -\BAR{X} \end{pmatrix}
    \;;\;
    \begin{array}{l}
    X: \text{$m\times m$ skew Hermitian}, \\
    Y \in \Alt_m(\CC)
    \end{array}
    \right\}
  \subset \frakg_0,
\end{gather*}
and we have the direct sum decomposition 
$\frakg_0 = \frakk_0 \oplus \frakp_0$.
The real Lie algebras $\frakg_0$ and $\frakm_0$
are embedded into $\fraksp_{4mn}(\RR)$ as follows:
\begin{align*}
&
\begin{array}{cccccc}
\frakg_0 &=& \fraku_{2m} & \hookrightarrow & \fraksp_{4mn}(\RR) \\
&& A + \sqrt{-1}B &\mapsto &
\begin{pmatrix}
  A \ast 1_n & -B \ast 1_n \\ B \ast 1_n & A \ast 1_n
\end{pmatrix}\
\end{array}
\quad (A, B: \text{ real matrices}),
\\
%
& {
\newcommand{\twoxtwomat}[4]{%
\Bigl(\begin{array}{@{\,}c@{\,}c@{\,}} {#1} & {#2}  \\[-2pt] {#3} & {#4} \end{array}\Bigr)}
\begin{array}{cccccc}
\frakm_0 &=& \frako^\ast_{2n} & \hookrightarrow & \fraksp_{4mn}(\RR) \\
&& 
\begin{pmatrix} X & -Y \\ \BAR{Y} & \BAR{X} \end{pmatrix} 
& \mapsto&
\begin{pmatrix}
\twoxtwomat{X_1}{-Y_1}{Y_1}{X_1}{\vphantom{\bigr)}}^{\oplus m}
&
\twoxtwomat{-X_2}{-Y_2}{Y_2}{-X_2}{\vphantom{\bigr)}}^{\oplus m}
\\
\twoxtwomat{X_2}{-Y_2}{Y_2}{X_2}{\vphantom{\bigr)}}^{\oplus m}
&
\twoxtwomat{X_1}{Y_1}{-Y_1}{X_1}{\vphantom{\bigr)}}^{\oplus m}
\end{pmatrix}
\end{array}
},
\end{align*}
where we write 
$ X = X_1 + \imgunit X_2 $ and 
$ Y = Y_1 + \imgunit Y_2 $ 
with real matrices $ X_i, Y_i \; ( i = 1, 2 ) $.

Let $V = \Mat(2m, n; \CC)$, and 
$x_{s,i}$ ($1 \le s \le 2m, 1 \le i \le n$) its linear coordinates,
and put $\partial_{s,i} = \partial/\partial x_{s,i}$.
Denote the Weil representation of $\fraksp_{4mn}(\CC)$ 
on the polynomial ring $\CC[V]$ by $\omega$.
%
The action of $\omega$ for the basis element 
$ E_{s,t} \in \frakg = \frakgl_{2m}(\CC)$ is given by
(see (4.5) of \bracketcite{MR1845714}, e.g.)
\begin{align*}
  \omega(E_{s,t}) 
  &= 
\textstyle
  \sum_{i=1}^n x_{s,i} \partial_{t,i} + \frac{n}{2} \delta_{s,t}
  \qquad
  (1 \le s,t \le 2m),
\end{align*}
%
while the action of $\frakm = \frako_{2n}(\CC)$ is
\begin{align*}
  \omega\bigl(
    \begin{pmatrix} \scriptstyle E_{i,j} & 0 \\ 0 & \scriptstyle -E_{j,i} \end{pmatrix}
  \bigr)
  &=
\textstyle
  \sum\limits_{s=1}^{2m} x_{s,i} \partial_{s,j}  + m \delta_{i,j},
  \\
  \omega\bigl(
    \begin{pmatrix} 0 & \scriptstyle E_{i,j}-E_{j,i} \\ 0 & 0 \end{pmatrix}
  \bigr)
  &= \imgunit 
\textstyle
  \sum\limits_{s=1}^m
  (x_{s,i} x_{\BAR{s},j} - x_{\BAR{s},i} x_{s,j}),
  \\
  \omega\bigl(
    \begin{pmatrix} 0 & 0 \\ \scriptstyle E_{j,i}-E_{i,j} & 0 \end{pmatrix}
  \bigr)
  &=
  \imgunit 
\textstyle
  \sum\limits_{s=1}^m
  (\partial_{s,i} \partial_{\BAR{s},j} - 
    \partial_{\BAR{s},i} \partial_{s,j}),
\end{align*}
where $\BAR{s} = s+m$,
and $E_{i,j}$ or $E_{j,i}$ denotes the matrix unit in $\Mat(n ; \CC)$.
Note that we used the normalization 
$x_{s,i} \mapsto \sqrt{2} x_{s,i}$ and
$x_{\BAR{s},i} \mapsto -\sqrt{-2} x_{\BAR{s},i}$
($1 \le s\le m, 1 \le i \le n$)
to (4.5) of \bracketcite{MR1845714}.
Note also that $x_{s,i}$ of this article
corresponds to $x_{2n(s-1)+i}$ in \bracketcite{MR1845714},
and $x_{\BAR{s},i}$ of this article
corresponds to $x_{(2n+1)(s-1)+i}$ in \bracketcite{MR1845714}
for $1 \le s \le m$.
We set 
\begin{align*}
J &=  \begin{pmatrix} 0 & 1_m \\ -1_m & 0 \end{pmatrix}
\end{align*}
and define the following matrices:
\begin{align*}
X &=
(x_{s,i})_{\substack{1\le s \le 2m\\ 1\le i \le n} }
&
\partial &=
(\partial/\partial x_{s,i})_{\substack{1\le s \le 2m\\ 1\le i \le n} }
&& \in \Mat(2m,n,\PD(V)),
\\
P &= (X, J\partial) 
&
Q &= (\partial, J X) 
&& \in \Mat(2m, 2n; \PD(V)),
\\
\EE &= (E_{s,t})_{1 \le s,t \le 2m}
&&
&& \in \Mat(2m; U(\frakgl_{2m}(\CC)),
\end{align*}
and we use the notation like $\omega(\EE)$
of matrix form.   
For example, we write
\begin{align*}
\omega(\EE + J \T\EE J^{-1}) = P \T Q.
\end{align*}

\subsection{Capelli identity for Case $\HH$}

Recall that the set $\comb{m}{d}$ ($1 \le d \le m$) of strictly increasing indices.
Here we define $\bcomb{m}{d} \; (m \ge 1, \; d \ge 1) $ as the set of 
\emph{weakly} increasing indices:
\begin{align*}
  \bcomb{m}{d} = \left\{ 
    S = \{s_1, s_2, \ldots, s_d\}
    \;;\;
    \begin{array}{l}
      \text{$S$ is a multi-set with} \\
      1 \le s_1 \le s_2 \le \cdots \le s_d \le m
    \end{array}
   \right\} .
\end{align*}
For $S \in \bcomb{m}{d}$,
we define an integer $S!$ by
\begin{align*}
  S! = t_1! \, t_2! \, \cdots t_m! \, ,
\end{align*}
where $t_j$ denotes the number of $j$ occurring in the multi-set $S$ 
($ \{ t_j \} $ are \emph{not} the members of $ S $).
For an $m\times m$ matrix $A$ and $I, J \in \bcomb{m}{d}$,
we write $A_{I,J} = (A_{i_a,j_b})_{1 \le a,b \le d}$.
Note that $A_{I,J}$ is not necessarily a submatrix of $A$ in this case, 
since $ I $ and $ J $ might have duplicated indices.  

We need the permanent (column-permanent)
and the symmetrized permanent.
\begin{definition}
\label{defn:pers}
For an $n \times n$ matrix $A$,
the \emph{permanent} of $A$ is defined by
\begin{align*}
\per A &= \sum_{\sigma\in\Symm_n}
A_{\sigma(1),1} A_{\sigma(2),2} \cdots A_{\sigma(n),n}.
\end{align*}
We define the permanent with diagonal parameters $ u = ( u_1 , u_2 , \dots , u_d ) $, 
and the \emph{symmetrized permanent} by 
\begin{align*}
  \per(A_{I,J}; u)
  &= 
\textstyle
  \sum_{\sigma\in\Symm_d}
  (A_{i_{\sigma(1)},j_1} + u_1 \delta_{i_{\sigma(1)},j_1})
  \cdots
  (A_{i_{\sigma(d)},j_d} + u_d \delta_{i_{\sigma(d)},j_d})
  \\ &=
  \per( A_{I,J} + 1_{I,J}
    \begin{smallpmatrix}
      u_1     &   0       &  \cdots  &   0  \\
      0       &   u_2     &  \cdots  &   0  \\
      \vdots  &&  \ddots  &  \vdots  \\
      0       &   0       &  \cdots  &   u_d
    \end{smallpmatrix}
  )
  \qquad (I, J \in \bcomb{n}{d}),
  \\
  \Per A &= \frac{1}{n!} 
\textstyle
  \sum_{\sigma,\tau\in\Symm_n}
  A_{\sigma(1),\tau(1)} A_{\sigma(2),\tau(2)} \cdots
  A_{\sigma(n),\tau(n)}.
\end{align*}
\end{definition}
\begin{remark}
It can be shown that
\begin{align*}
\textstyle
  \sum_{S\in\bcomb{2m}{d}} \dfrac{1}{S!}
  \per( \EE + J\T\EE J^{-1} )_{S,S} 
  \in S(\frakp)^K
  \quad (d = 1, 2, \ldots)
\end{align*}
is a $K$-invariant element in $S(\frakp)$,
where $\EE$ is considered as a matrix with entries
$E_{s,t}$ in $S(\frakg)$.
Note that these elements do \textit{not} generate $S(\frakp)^K$,
but elements constructed by using Pfaffians generate $S(\frakp)^K$.  
It seems that Capelli identities corresponding to the elements
using Pfaffians are rather difficult to treat.
This phenomenon occurs in the case of 
the ordinary Capelli identities for the orthogonal Lie algebras 
(cf.~\bracketcite{MR1116239, MR2001d:17009, MR1912651}).
\end{remark}

We have two different types of the Capelli identities for Case $\HH$.
\begin{theorem}
\label{thm:caseH-easy}
For $d \ge 1$, define
$X_d \in U(\frakg)^K = U(\frakgl_{2m}(\CC))^{USp_m}$ by
\begin{align*}
&  X_d = 
\textstyle
 \sum\limits_{S\in\bcomb{2m}{d}}
  \sum\limits_{l=0}^d 
  \sum\limits_{\substack{
    S' , S\dblprime \\
    T' , T\dblprime }}
  \dfrac{S!}{S'! \, S\dblprime! \, T'! \, T\dblprime!}
  \per( (\EE - \frac{n}{2}1_{2m})_{S',T'}; \rho'_l )
  \\[-5pt] & \hspace{25ex}
  \times \per( (J\T\EE J^{-1} + \frac{n}{2}1_{2m})_{S\dblprime,T\dblprime}; \rho''_{d -l} ),
\\ &
\rho'_l = ( -(l-1), -(l-2), \ldots, 0 ) , \quad 
\rho''_{d - l} = ( d-l-1, d-l-2, \ldots, 0 )
\end{align*}
where the third summation is taken over 
$ S' , T' \in \bcomb{2m}{l} $ and 
$ S\dblprime , T\dblprime \in \bcomb{2m}{d - l} $ 
which satisfies 
$ S = S' \cup S\dblprime = T' \cup T\dblprime $.
%
Then we have the Capelli identity:
\begin{align*}
\omega(X_d) = 
\textstyle 
  \sum_{S\in\bcomb{2m}{d}}
  \sum_{J\in\bcomb{2n}{d}}
  \dfrac{1}{S! \, J!}
  \per P_{S,J} \per Q_{S,J}.
\end{align*}
\end{theorem}

\begin{theorem}
\label{thm:caseH-not-easy}
For $d \ge 1$, define
$X_d \in U(\frakg)^K = U(\frakgl_{2m}(\CC))^{USp_m}$ by
\begin{align*}
  X_d = 
\textstyle 
  \sum_{S \in \bcomb{2m}{d}}
  \dfrac{1}{S!} \Per(\EE + J \T\EE J^{-1})_{S,S} .
\end{align*}
Then we have the Capelli identity:
\begin{align*}
\textstyle 
\omega(X_d) = 
\sum\limits_{l=0}^d \dfrac{(2m+d-1)!}{d!(2m+l-1)}
\sum\limits_{J\in\bcomb{2n}{l}} \dfrac{c^d_{\beta,\alpha}}{J!}
\sum\limits_{S\in\bcomb{2m}{l}} \dfrac{1}{S!}
\per P_{S,J} \per Q_{S,J},
\end{align*}
where $c^d_{\beta,\alpha}$ is an integer defined in 
Equation~\eqref{eq:def-of-c^d_J} 
with both $p$ and $q$ replaced by $n$,
and $\alpha$ and $\beta$ are determined from $J$
in the same way as in Equation~\eqref{eq:def-of-c^d_J-alpha-beta}.
\end{theorem}

\section*{Acknowledgment}
\label{Acknowledgiment}
We thank Yasuhide Numata
for proving Lemmas~\ref{lem:(19.2)} and \ref{lem:(19.4)} below.
We also thank Hiroshi Oda
for fruitful discussions on the identities. 
The first author is partially supported by JSPS Grant-in-Aid for Scientific Research (B) \#{17340037}.

\appendix
\section{Formula for $c^d_J$}
\label{sec:c^d_J}

In this section we prove (Lemma~\ref{lem:(19.4)})
\begin{align*}
c^d_{\alpha,\beta}
=
\sum_{u+v=\alpha+\beta} b^d_{u,v} \, \veps(\alpha,\beta;u,v),
\end{align*}
where
$c^d_{\alpha,\beta}$,
$b^d_{u,v}$ 
and $\veps(\alpha,\beta;u,v)$
are defined in
Equation~(\ref{eq:def-of-c^d_J}),
Definition~\ref{defn:b-d-uv}
and Lemma~\ref{lem:gamma-to-det},
respectively.
We start with a simplification of $\veps(\alpha,\beta;u,v)$.
\begin{lemma}
\label{lem:simplify-e(J;u,v)}
For non-negative integers $p,q,u,v$ and $J \in \comb{p+q}{u+v}$, 
the integer 
$\veps(\alpha, \beta; u, v)$ is defined in 
Lemma~\ref{lem:gamma-to-det}:
\begin{align*}
  \veps(J; u, v) =
  \veps(\alpha, \beta; u, v) &= 
\textstyle 
  \sum_{\sigma\in\Symm_{u+v}}
  \veps_{j_{\sigma(u+1)},j_{\sigma(u+1)}} \cdots
  \veps_{j_{\sigma(u+v)},j_{\sigma(u+v)}},
\end{align*}
where $\alpha$ and $\beta$ are defined by
Equation~\eqref{eq:def-of-c^d_J-alpha-beta}.
We have 
\begin{align*}
 \veps(\alpha, \beta; u, v) =
\textstyle 
  u! \, v! \sum_{\substack{a,b\in\ZZ\\a+b=v}} (-1)^a
\displaystyle
  \binom{\alpha}{a} \binom{\beta}{b}.
\end{align*}
\end{lemma}
\begin{proof}
For a summand 
$\veps_{j_{\sigma(u+1)},j_{\sigma(u+1)}} \cdots
\veps_{j_{\sigma(u+v)},j_{\sigma(u+v)}}$
of $\veps(\alpha,\beta;u,v)$,
let $a$ be the number of $j_{\sigma(u+i)}$ ($1 \le i \le v$)
which is less than or equal to $p$.
Similarly let $b$ be the number of $j_{\sigma(u+i)}$ ($1 \le i \le v$)
which is greater than or equal to $p+1$.
In particular $a+b = v$, 
and the value of the summand is equal to $(-1)^a$.

We count the number of $\sigma$'s in $\Symm_{u+v}$
such that they give the same $a$ and $b$.
There are $\binom{\alpha}{a} \binom{\beta}{b}$ choices of subset 
$\{ \sigma(u+1), \sigma(u+2), \ldots, \sigma(u+v) \}$
of $\{ 1,2,\ldots,u+v \}$.
By considering the order of this subset and 
the the order of the complement
$\{ \sigma(1), \sigma(2), \ldots, \sigma(u) \}$,
the number is multiplied by $u!v!$.
Thus the number of such $\sigma$ is equal to
$\binom{\alpha}{a} \binom{\beta}{b} u! \, v!$.
\end{proof}

Next we show 
a recurrence formula for $b^d_{u,v}$ in Lemma~\ref{lem:(12.2)}
and two closed formulas in Lemmas~\ref{lem:(12.12)=(19.1)} and \ref{lem:(19.2)}.
\begin{lemma}
\label{lem:(12.2)}
We have a recurrence formula for $b^d_{u,v}$:
\begin{equation*}
b^0_{0,0} = 1 , \quad 
b^{d+1}_{u,v} =
b^{d}_{u-1,v} +
(u+1) b^{d}_{u+1,v-1} +
(v+1) b^{d}_{u-1,v+1} .
\end{equation*}
\end{lemma}
\begin{proof}
By the definition of $b^d_{u,v}$, we have
\begin{align*}
\Xi^{d+1} 
&= 
\sum_{0 \le u+v \le d} 
b^d_{u,v} \gamma(u,v) \tau^{d-u-v} \Xi.
\end{align*}
It follows from Lemma~\ref{lem:gamma-recurrence}~(1) that
the expression above is equal to
\begin{align*}
\sum_{0 \le u+v \le d} 
b^d_{u,v} \tau^{d-u-v} 
\left(
\gamma(u+1,v) + u \tau \gamma(u-1,v+1) + v \tau \gamma(u+1,v-1)
\right).
\end{align*}
Hence we have
$
b^{d+1}_{u,v} 
=
b^{d}_{u-1,v} + (u+1) b^{d}_{u+1,v-1} + (v+1) b^{d}_{u-1,v+1}.
$
\end{proof}
\begin{lemma}
\label{lem:(12.9)}
We define the generating function $f_d$ of $b^d_{u,v}$ by 
\begin{math}
f_d(x,y) = \sum_{u,v\in\ZZ} b^d_{u,v} x^u y^v.
\end{math}
Then it satisfies 
\begin{math}
f_{d+1}(x,y) = 
x f_d 
+ y\frac{\partial}{\partial x} f_d 
+ x \frac{\partial}{\partial y} f_d.
\end{math}
\end{lemma}
\begin{proof}
By multiplying $x^u y^v$ to the formula of Lemma~\ref{lem:(12.2)},
and take a sum over $u,v\in\ZZ$,
and we have
\begin{align*}
f_{d+1} &=
\textstyle
\sum_{u,v\in\ZZ} b^d_{u-1,v}  x^u y^v
+ \sum_{u,v\in\ZZ} (u+1) b^d_{u+1,v-1}  x^u y^v \\
& \hspace*{.4\textwidth}
\textstyle
+ \sum_{u,v\in\ZZ} (v+1) b^d_{u-1,v+1}  x^u y^v
\\[-5pt] 
&=
x f_d 
+ y\frac{\partial}{\partial x} f_d 
+ x \frac{\partial}{\partial y} f_d.
\end{align*}
\end{proof}
\begin{lemma}
\label{lem:(12.12)=(19.1)}
We have a closed formula for $b^d_{u,v}$:
\begin{align*}
  b^d_{u,v} &=
  \sum_{m \ge 0, \mu,\nu\in\ZZ}
  \frac{(u-v+m+2\mu-2\nu)^d (-1)^{m+\nu}}{2^{u+v-m}u!v!}
  \binom{u}{\nu} \binom{v}{m} \binom{v-m}{\mu} .
\end{align*}
\end{lemma}
\begin{proof}
Set $\varphi_d = e^y f_d$.
Since we have an identity 
\begin{math}
\frac{\partial}{\partial y} e^{-y} = 
e^{-y} \frac{\partial}{\partial y} - e^{-y} 
\end{math}
as differential operators,
it follows from Lemma~\ref{lem:(12.9)} that
\begin{align*}
e^{-y} \varphi_{d+1} 
&= 
e^{-y} x \varphi_d
+ e^{-y} y \frac{\partial}{\partial x} \varphi_d
+ x \left( e^{-y} \frac{\partial}{\partial y} - e^{-y} \right)
\varphi_d,
\end{align*}
and hence
\begin{align}
\label{eqn:euler.operator}
\varphi_{d+1} =
y \frac{\partial}{\partial x} \varphi_d +
x \frac{\partial}{\partial y} \varphi_d.
\end{align}
Put $x=a-b$, $y=a+b$.
Then we have
\begin{align*}
y \frac{\partial}{\partial x} + x \frac{\partial}{\partial y}
=
a \frac{\partial}{\partial a} - b \frac{\partial}{\partial b}
=
\theta_a - \theta_b , 
\quad
\text{where $ \theta_a = a \frac{\partial}{\partial a} $. }
\end{align*}
If we make the change of variables and put $\psi_d(a,b) = \varphi_d(a-b,a+b)$, 
Equation 
\eqref{eqn:euler.operator} 
gives
\begin{align*}
\psi_{d} = (\theta_a - \theta_b) \psi_{d - 1}
= (\theta_a - \theta_b)^d \psi_0
= (\theta_a - \theta_b)^d e^{a+b} 
\quad (\because \psi_0 = e^{a+b}).
\end{align*}
Since $(\theta_a - \theta_b)(a^k b^l) = (k-l)a^k b^l$,
we have
\begin{align*}
\psi_d 
&= 
\textstyle
(\theta_a - \theta_b)^d \left(
  \sum_{k,l\ge 0} \dfrac{a^k b^l}{k! \, l!}
\right)
=
\sum_{k,l\ge 0} \dfrac{(k-l)^d}{k! \, l!} a^k b^l.
\end{align*}
Therefore
\begin{align*}
f_d &= e^{-y} \psi_d
= e^{-y}
\sum_{k,l\ge0}
\frac{(k-l)^d}{k! \, l!} 
\left( \frac{x+y}{2} \right)^k
\left( \frac{y-x}{2} \right)^l
\\ &=
\sum_{k,l,m \ge 0}
\frac{(-y)^m}{m!} \cdot
\frac{(k-l)^d}{k! \, l!} 
\sum_{\mu,\nu\in\ZZ}
\binom{k}{\mu} \frac{x^{k-\mu} y^\mu}{2^k} 
\binom{l}{\nu} \frac{y^{l-\nu} (-x)^\nu}{2^l} 
\\ &=
\sum_{\substack{k,l,m \ge 0 \\ \mu,\nu\in\ZZ}}
\frac{(k-l)^d (-1)^{m+\nu}}{2^{k+l} k! \, l! \, m!}
\binom{k}{\mu} \binom{l}{\nu}
x^{k-\mu+\nu} y^{m+\mu+l-\nu},
\\
\intertext{(let $u=k-\mu+\nu$ and $v=m+\mu+l-\nu$)}
&=
\!\! \sum_{\substack{u,v,m\ge0 \\ \mu,\nu\in\ZZ}} \!\! 
\frac{(u-v+m+2\mu-2\nu)^d (-1)^{m+\nu}}{ 2^{u+v-m} m! }
\cdot
\frac{x^u y^v}{(u-\nu)! \, \mu! \, (v-m-\mu)! \, \nu!}
\\ &=
\!\! \sum_{\substack{u,v,m\ge0 \\ \mu,\nu\in\ZZ}} \!\! 
\frac{(u-v+m+2\mu-2\nu)^d (-1)^{m+\nu}}{ 2^{u+v-m} m! }
\binom{u}{\nu} \binom{v-m}{\mu} \binom{v}{m}
\frac{x^u y^v}{u! v!}.
\end{align*}
\end{proof}
\begin{lemma}
\label{lem:(19.2)}
We have another closed formula 
\begin{align*}
b^d_{u,v}
&=
\frac{1}{2^{u+v} u! \, v!}
\sum_{k,l\in\ZZ}
\binom{2v}{k} \binom{u}{l} 
(-1)^{k+l} (u+v-k-2l)^d.
\end{align*}
\end{lemma}
\begin{proof}
Put 
\begin{align*}
F^d_{u,v}(T) &= 
\frac{1}{2^{u+v} u!v!} T^{u-v}
(-2T+T^2+1)^v (1-T^{-2})^u.
\end{align*}
By the multinomial expansion, we have
\begin{align*}
F^d_{u,v} (T)
&=
\frac{1}{2^{u+v} u! \, v!}
\sum_{m,\mu,\nu\in\ZZ}
\binom{v}{m,\mu} \binom{u}{\nu}
(-1)^{m+\nu} 2^m
T^{u-v+m+2\mu-2\nu}.
\end{align*}
Let $\theta = T (\partial/\partial T)$, 
and apply $\theta^d$ to the expression above and then substitute $ T = 1 $. Then we have
\begin{align*}
&\theta^d( F^d_{u,v} ) \Bigm|_{T = 1} \\
&=
\frac{1}{2^{u+v} u!v!}
\sum_{m,\mu,\nu\in\ZZ}
\binom{v}{m,\mu} \binom{u}{\nu}
(-1)^{m+\nu} 2^m
(u-v+m+2\mu-2\nu)^d , 
\end{align*}
which is equal to $b^d_{u,v}$ by Lemma \ref{lem:(12.12)=(19.1)}.  
On the other hand 
\begin{align*}
F^d_{u,v}(T)
&=
\frac{T^{u-v}}{2^{u+v} u!v!} (T-1)^{2v} (1-T^{-2})^u
=
\frac{T^{-u-v}}{2^{u+v} u!v!} (T-1)^{2v} (T^2-1)^u
\\ &=
\frac{1}{2^{u+v} u!v!} 
\sum_{k,l\in\ZZ} 
\binom{2v}{k} \binom{u}{l} (-1)^{k+l}
T^{u+v-k-2l}.
\end{align*}
Similar computation as above, i.e., 
applying $\theta^d$ to this expression and substituting $T=1$,
gives us the same value as 
\begin{align*}
\theta_d(F^d_{u,v}) \Big|_{T=1}
&=
\frac{1}{2^{u+v} u! \, v!} 
\sum_{k,l\in\ZZ} 
\binom{2v}{k} \binom{u}{l} (-1)^{k+l}
(u+v-k-2l)^d.
\end{align*}
\end{proof}

By Lemmas~\ref{lem:(19.2)} and \ref{lem:simplify-e(J;u,v)},
we have
\begin{align}
  \nonumber
  & 
\textstyle
  \sum\limits_{u+v=\alpha+\beta} b^d_{u,v} \veps(\alpha,\beta;u,v)
  \\ & \quad =
  \nonumber
\textstyle
  \sum\limits_{u+v=\alpha+\beta}
  \dfrac{1}{2^{u+v} u! \, v!}
  \sum\limits_{k,l\in\ZZ}
  {\displaystyle \binom{2v}{k} \binom{u}{l}}
  (-1)^{k+l} (u+v-k-2l)^d \\[-5pt]
& \nonumber
\hspace*{.4\textwidth} \times 
\textstyle
  \sum\limits_{a+b=v} u! \, v! \, (-1)^a
  {\displaystyle \binom{\alpha}{a} \binom{\beta}{b}}
  \\[3pt] \nonumber
   & 
  \quad =
  \frac{1}{2^{\alpha+\beta}}
\textstyle
  \sum\limits_{k,l,a,b\in\ZZ}
  {\displaystyle \binom{\alpha}{a} \binom{\beta}{b}
  \binom{2a+2b}{k} \binom{\alpha+\beta-a-b}{l} } 
\\
& \hspace*{.4\textwidth} \times
  (-1)^{k+l+a} (\alpha+\beta-k-2l)^d.
  \label{eq:(19.3)}
\end{align}
Now we are ready to show the following lemma which is the goal of this section.
\begin{lemma}
\label{lem:(19.4)}
We have
\begin{align*}
c^d_{\alpha,\beta} =
  \sum_{u+v=\alpha+\beta} b^d_{u,v} \veps(\alpha,\beta;u,v).
\end{align*}
\end{lemma}
\begin{proof}
Set 
\begin{align*}
F^d_{\alpha,\beta}(T)
&=
\left( \frac{T^2-1}{2T} \right)^{\alpha+\beta} 
\bigl( - \dfrac{(T-1)^2}{(T^2-1)} +1 \bigr)^\alpha
\bigl( \dfrac{(T-1)^2}{(T^2-1)} + 1 \bigr)^\beta.
\end{align*}
First, by the binomial expansion, we have
\begin{align*}
F^d_{\alpha,\beta}(T) 
&=
\left( \frac{T^2-1}{2T} \right)^{\alpha+\beta}
\sum_{a\in\ZZ}
\binom{\alpha}{a} (-1)^a \dfrac{(T-1)^{2a}}{(T^2-1)^a}
\sum_{b\in\ZZ}
\binom{\beta}{b} \dfrac{(T-1)^{2b}}{(T^2-1)^b}
\\ &=
\frac{1}{(2T)^{\alpha+\beta}}
\sum_{a,b\in\ZZ} 
\binom{\alpha}{a} \binom{\beta}{b} (-1)^a
(T-1)^{2a+2b} (T^2-1)^{\alpha+\beta-a-b}
\\ &=
\frac{1}{(2T)^{\alpha+\beta}}
\sum_{a,b\in\ZZ} 
\binom{\alpha}{a} \binom{\beta}{b} (-1)^a
\sum_{k\in\ZZ}
\binom{2a+2b}{k} T^{2a+2b-k} (-1)^k
\\ & \hspace{10ex}
\times
\sum_{l\in\ZZ}
\binom{\alpha+\beta-a-b}{l} T^{2(\alpha+\beta-a-b-l)} (-1)^l
\\ &=
\frac{1}{2^{\alpha+\beta}}
\sum_{a,b,k,l\in\ZZ} 
\binom{\alpha}{a} \binom{\beta}{b} 
\binom{2a+2b}{k} \binom{\alpha+\beta-a-b}{l} 
\\ & \hspace{20ex}
\times
(-1)^{a+k+l}
T^{\alpha+\beta-k-2l}.
\end{align*}
We apply $\theta^d$ to this expression, and let $T=1$.
Then we have
\begin{align*}
\theta^d (F^d_{u,v}) \Big|_{T=1} = 
\sum_{u+v=\alpha+\beta} b^d_{u,v} \veps(\alpha,\beta;u,v),
\end{align*}
by Equation~(\ref{eq:(19.3)}).
Second, we proceed as 
\begin{align*}
F^d_{\alpha,\beta} (T)
&=
\frac{1}{(2T)^{\alpha+\beta}}
( -(T-1)^2 + (T^2-1) )^\alpha
( (T-1)^2 + (T^2-1) )^\beta
\\ &=
\frac{1}{(2T)^{\alpha+\beta}}
(2T-2)^\alpha (2T^2 - 2T)^\beta
\\ &=
\frac{(T-1)^{\alpha+\beta}}{T^\alpha}
= 
\sum_{k\in\ZZ}
\binom{\alpha+\beta}{k}
T^{\alpha+\beta-k}
(-1)^k   
\cdot
\frac{1}{T^\alpha}.
\end{align*}
We apply $\theta^d$ to this expression, and let $T=1$.
Then we have
\begin{align*}
\theta^d (F^d_{u,v}) \Big|_{T=1} = 
c^d_{\alpha,\beta},
\end{align*}
by the definition of $c^d_{\alpha,\beta}$ 
(Equation~\eqref{eq:def-of-c^d_J}).
\end{proof}


\bibliographystyle{amsplain}

\def\cprime{$'$} \def\Dbar{\leavevmode\lower.6ex\hbox to 0pt{\hskip-.23ex
  \accent"16\hss}D} \def\polhk#1{\setbox0=\hbox{#1}{\ooalign{\hidewidth
  \lower1.5ex\hbox{`}\hidewidth\crcr\unhbox0}}}
\providecommand{\bysame}{\leavevmode\hbox to3em{\hrulefill}\thinspace}

\end{document}